\documentclass[12pt]{amsart}

\usepackage{amssymb,latexsym,amsmath,amsthm,
graphicx,titletoc,color,palatino,mathpazo, amsfonts}
\usepackage{lipsum}
\usepackage{amsfonts}
\usepackage{graphicx}
\usepackage{epstopdf}
\usepackage{algorithmic}

\numberwithin{equation}{section}
\newtheorem{theorem}{Theorem}[section]
\newtheorem{proposition}[theorem]{Proposition}
\newtheorem{lemma}[theorem]{Lemma}

\theoremstyle{definition}

\def\XXint#1#2#3{{\setbox0=\hbox{$#1{#2#3}{\int}$}
     \vcenter{\hbox{$#2#3$}}\kern-.5\wd0}}

\def\p{\partial}

\def\R{\mathbb R}

\newcommand{\ov}{\overline}

\newcommand{\ve }{\varepsilon }

\newcommand{\m}{\mathcal}

\begin{document}

 \author{Weiwei Ao \and Yehui Huang  \and Yong Liu \and Juncheng Wei}

\title[Multi vortex ring solutions]{Generalized Adler-Moser Polynomials and Multiple vortex rings for the Gross-Pitaevskii equation}

\date{\today}

\maketitle

\begin{abstract}
	New finite energy traveling wave solutions with small speed are constructed for the three dimensional Gross-Pitaevskii equation
	\begin{equation*}
	i\Psi_t=  \Delta \Psi+(1-|\Psi|^2)\Psi,
	\end{equation*}
	where $\Psi$ is a complex valued function defined on ${\mathbb R}^3\times{\mathbb R}$. These solutions have the shape of $2n+1$ vortex rings, far away from each other. Among these vortex rings,  $n+1$ of them have positive orientation and the other $n$ of them have negative orientation. The location of these rings are described by the roots of a sequence of polynomials with rational coefficients. The  polynomials found here can be regarded as a generalization of the classical Adler-Moser polynomials and can be expressed as the Wronskian of certain very special functions. The techniques used in the derivation of these polynomials should have independent interest.
\end{abstract}

\section{Introduction}\label{sec1}


In this paper, we are interested in the existence of solutions with the shape of multiple vortex rings, to
the nonlinear Schr\"odinger type problem
\begin{eqnarray}\label{GrossPitaevskii}
i\Psi_t\,= \triangle \Psi+\Big(1-|\Psi|^2\Big)\Psi,
\end{eqnarray}
where $\triangle=\partial^2_{y_1}+\partial^2_{y_2}+\partial^2_{y_3}$
is the Laplacian operator in ${\mathbb R}^3$.
Equation (\ref{GrossPitaevskii}), usually called Gross-Pitaevskii equation (GP),
is a well-known mathematical model arising in various physical contexts such as nonlinear optics and Bose-Einstein condensates, see for instance  \cite{pethicksmith}.

Traveling wave solutions of the GP equation play important role in its long time dynamics. If $\Psi$ is a traveling wave type solution of the form
$$
\Psi(y,t)
\,=\,
{\tilde u}\big(y_1,\,y_2,\, y_3- ct\big),
$$
then ${\tilde u({\tilde y}_1,{\tilde y}_2,{\tilde y}_3)}$ will be a solution of the nonlinear elliptic problem
\begin{eqnarray}\label{TWK0}
\,-\,i\,c\,\frac{\partial\tilde u}{\partial {\tilde y}_3}
\,=\,
\triangle{\tilde u}
\,+\,
\Big(1-|{\tilde u}|^2\Big){\tilde u}.
\end{eqnarray}

The existence or nonexistence of traveling wave solutions to (\ref{TWK0}) with $\tilde{u}\to 1$ as $|{\tilde y}|\to \infty$ has attracted much attention in the literature, initiated from the work of Jones, Putterman,
Roberts \cite{J,J1}, where they studied the equation from the physical point of view and obtained solutions with formal and numerical calculation. They carried out their computation in dimension two and three, and find that the solution branches  in these two cases have different properties. In particular, in the energy-momentum diagram, the branch in 2D is smooth, while the branch in 3D has a cusp singularity. In any case, the solutions they found have traveling wave speed $c$ less than $\sqrt{2}$ (the sound speed in this context, appears after taking the Madelung transform for the GP equation).

A natural question is whether there exist solutions whose traveling speed is larger than the sound speed. In this respect, the nonexistence of finite energy solutions with $c>\sqrt{2}$
is rigorously proved by Gravejat in \cite{G2,G3}. This result is also true for $c=\sqrt{2}$ in $\mathbb{R}^2$, but the higher dimensional case is still open.

The first rigorous mathematical proof of the existence is carried out in \cite{Ber3}, where solutions in 2D with small traveling speed are obtained using mountain pass theorem. Later on, the existence of small speed solutions in dimension larger than two are proved in \cite{chiron}, also based on the mountain pass theorem. In \cite{Ber3}, a different approach, minimizing the action functional with fixed momentum, is applied to get the existence of solutions with large momentum in dimension $N\geq 3$. This method is further developed in \cite{Ber1} to all dimensions, yielding existence or nonexistence of solutions for any fixed momentum. The asymptotic profile of these solutions are also studied in the above mentioned papers. In particular, for $c$ close to $0,$ in 2D, these solutions have two vortice and around them, the solution is close to the degree one vortex solution of the the Ginzburg-Landau equation; while in 3D, the solutions have the shape of a single vortex ring, see also \cite{ChironP}. We also refer to the paper \cite{bethuelgravejatsaut1} by F. Bethuel, P. Gravejat and J. Saut and the references therein for more details and discussions.

The question of existence for all traveling speed $c\in (0, \sqrt{2})$ is quite delicate. It is proved by Maris in \cite{Mar} that in dimension $N>2$, one can minimize the action under a Pohozaev constraint, obtaining solutions in the full speed interval  $(0, \sqrt{2})$. Unfortunately, this argument breaks down in 2D, thus leaving the problem open in this dimension. Recently, Bellazzini-Ruiz \cite{Ruiz} proved that the existence of almost all subsonic speed in 2D, using mountain pass argument. They also recovered the results of Maris in 3D.

Note that
when the parameter $c=0,$ equation (\ref{TWK0}) reduces to the
Ginzburg-Landau equation:
\begin{equation}
\Delta u+u\left(  1-\left\vert u\right\vert ^{2}\right)  =0. \label{Landau}%
\end{equation}
In $\mathbb{R}^2$, for each $\tau\in\mathbb{Z}\backslash\left\{  0\right\}  ,$ it
is known that the Ginzburg-Landau equation $\left(  \ref{Landau}\right)  $ has
a degree $\tau$ vortex solution. In the polar coordinate, it has the form $S_{\tau}\left(  r\right)  e^{i\tau\theta
}$. The function $S_{\tau}$ is real valued and vanishes exactly at $r=0.$ It
satisfies%
\[
-S_{\tau}^{\prime\prime}-\frac{1}{r}S_{\tau}^{\prime}+\frac{\tau^{2}}{r^{2}}S_{\tau}%
=S_{\tau}\left(  1-S_{\tau}^{2}\right)  \text{ in }\left(  0,+\infty\right)  .
\]
This equation indeed has a unique solution $S_{\tau}$ with $S_{\tau}\left(  0\right)
=0$ and $S_{\tau}\left(  +\infty\right)  =1$ and $S_{\tau}^{\prime}\left(  r\right)
>0.$ See \cite{Fife,Her} for a proof.

Recently, based on the vortex solutions of the Ginzburg-Landau equation, multi-vortex traveling wave solutions to (\ref{TWK0}) were constructed in \cite{Liu-Wei} using Lyapunov-Schmidt reduction method. These solutions have $\frac{n(n+1)}{2}$ pairs of vortex-anti vortex configuration, where the location of the vortex points are determined by the roots of the Adler-Moser Polynomials. It is worth pointing out that the Adler-Moser polynomials arise naturally from the rational solutions of the KdV equation. We also mention that as $c$ tends to $\sqrt{2},$ a
suitable rescaled traveling waves will converge to solutions of the KP-I
equation, which is an important integrable system, see \cite{Ber0, chiron2}. Interestingly, the KP-I equation is actually a two dimensional generalization of the classical KdV equation. Hence in the context of GP equation, we see the KP-I equation in the transonic limit and KdV in the small speed limit. The inherent reason behind this phenomena is still to be explored. As a related result, we would like to mention that numerical simulation has been performed in \cite{chiron3} to illustrate the higher energy solutions of the GP equation.

Denote the degree $\pm1$ vortex solutions of the
Ginzburg-Landau equation $\left(  \ref{Landau}\right)  $ as
\[
v_{+}=e^{i\theta}S_{1}\left(  r\right)  ,v_{-}=e^{-i\theta}S_{1}\left(
r\right)  .
\]
To better explain our main result in this paper, let us recall the following result proved in \cite{Liu-Wei}, which provides a family of multi-vortex solutions in dimension $2$.
\begin{theorem}[\cite{Liu-Wei}]
	\label{main}In $\mathbb{R}^2$, for each $n\leq34,$ there exists $c_{0}>0,$ such that
	for all $c\in \left(  0,c_{0}\right)  ,$ the equation
	$\left(  \ref{TWK0}\right)  $ has a solution $u_{c}$, with
	\[
	u_{c}=\prod\limits_{k=1}^{n\left(  n+1\right)  /2}\left[
	v_{+}\left(  z-c ^{-1}p_{k}\right)  v_{-}\left(  z+c
	^{-1}p_{k}\right)  \right] +o\left(  1\right)  ,
	\]
	where $p_{k}$, $k=1,...,n\left(  n+1\right)  /2$ are roots of the
	Adler-Moser polynomials.
\end{theorem}


\medskip

In this paper, we construct new traveling waves for $c$ close to
$0$ in 3D. The solutions will have multiple vortex rings.  By our construction below, it turns out that the location of the vortex points are closely related to the following system (Balancing condition):

\begin{equation}
\left\{
\begin{array}
[c]{l}%
{\displaystyle\sum\limits_{j=1,j\neq k}^{m}}
\frac{1}{\mathbf{a}_{k}-\mathbf{a}_{j}}-%
{\displaystyle\sum\limits_{j=1}^{n}}
\frac{1}{\mathbf{a}_{k}-\mathbf{b}_{j}}=-n,\text{ for }k=1,...,m,\\
-%
{\displaystyle\sum\limits_{j=1,j\neq k}^{n}}
\frac{1}{\mathbf{b}_{k}-\mathbf{b}_{j}}+%
{\displaystyle\sum\limits_{j=1}^{m}}
\frac{1}{\mathbf{b}_{k}-\mathbf{a}_{j}}=-m, \text{ for }k=1,...,n.
\end{array}
\right.  \label{Balancecondition}%
\end{equation}Here $\mathbf{a}_j,j=1,...,m, $  $\mathbf{b}_{\ell},\ell=1,...,n$, are complex numbers in the $z=x_1+i x_2$ plane. The integer $m$  actually denotes the number of positively oriented vortex rings and $n$ denotes the number of negatively oriented ones. Moreover, the solvability of our original problem is related to the nondegeneracy of the linearized operator $dF$ of the map $F$ defined by (\ref{mapF}).

The make our construction possible, the solution $\mathbf{a}_j,\mathbf{b}_{\ell}$ to the system (\ref{Balancecondition}) has to satisfy some symmetric properties. We therefore introduce the following condition:
\medskip

\noindent($\mathcal{M}$). $m>n$. The points $\mathbf{a}_j,\mathbf{b}_{\ell}$, $j=1,...,m$, $\ell=1,...,n$ are all distinct. The set of points of $\{\mathbf{a}_1,...,\mathbf{a}_m\}$ and  $\{\mathbf{b}_1,...,\mathbf{b}_n\}$ are both  symmetric with respect to the $x_1$ axis.

We use Lyapunov-Schmidt reduction method to construct multi-vortex ring solutions. Our main result is the following:
\begin{theorem}\label{thm1}
	Suppose $\mathbf{a}_j,\mathbf{b}_{\ell}$, $j=1,...,m$, $\ell=1,...,n$ is a solution of (\ref{Balancecondition}) satisfying condition $\mathcal{M}$ and the linearized operator $dF$ of (\ref{mapF}) is non-degenerate at this solution in the sense defined in Section 5. Then for all $\varepsilon>0 $ sufficiently small, there exists an axially symmetric solution $u= u(\sqrt{y_1^2+y_2^2}, y_3)$ to equation (\ref{oure}), and $u$ has $m$ positively oriented vortex rings and  $n$ negatively oriented vortex rings. The distance of the vortex rings to the axis is of the order $O(\varepsilon^{-1})$, while the mutual distance of two vortex rings are of the order $O(\varepsilon^{-1}/|\ln \varepsilon)|$. After scaling back by the factor $\varepsilon |\ln \varepsilon|$, the position of the vortex rings in the $(x_1,x_2)$ plane is close to suitable $x_1$-translation of those points $\{\mathbf{a}_1,...,\mathbf{a}_m,\mathbf{b}_1,...,\mathbf{b}_n\}$.
\end{theorem}

More precise description of the solutions can be found in the course of the proof. From the proof in Section \ref{sec5}, we can see that in the case of two positively oriented vortex rings and one negatively oriented vortex ring, there exists solutions to (\ref{Balancecondition}) and the corresponding linearized operator of (\ref{mapF}) is non-degenerate. Hence one can construct traveling wave solutions with three vortex rings. We also show in Section \ref{sec6} and Section 6 that  (\ref{Balancecondition}) has solutions satisfying $\mathcal{M}$, provided that $m=n+1$. (Surprisingly, if $m-n>1$, we have not found any solutions satisfying $\mathcal{M}$.)  When $m=n+1$ the location of the vortex points are determined by the roots of  generating polynomials which  have recurrence
relations  and can be explicitly written down
using certain Wronskians. These generating polynomials are natural generalizations of the classical Adler-Moser polynomials. We refer to Section 6 for more details.

Let us point out that traveling wave solutions of the Schrodinger map equation with single vortex ring has been constructed in \cite{LFH}. In principle, our method in this paper can also be applied to this equation and other related equation such as the Euler equation.

The dimension three case (with obvious extension to higher dimensions) studied in the present paper actually has some new properties compared to the 2D case. Roughly speaking, the main difference of the 2D and 3D case is the following. In 2D, the vortex location of our solutions is determined by the Adler-Moser polynomials. These polynomials can be obtained by method of integrable systems and are well studied. However, in 3D, due to the presence of additional terms in the equation, the vortex location is not determined by Adler-Moser polynomials. Indeed they are determined by a sequence of polynomials, which can be regarded as a generalization of Adler-Moser polynomials, and up to our knowledge, are new. We have to find these new generating polynomials using some techniques from the theory of integrable systems. This step is nontrivial and may have independent interest.

If we rescale the Gross-Pitaevskii equation $x=\epsilon^{-1} \bar{x}$, then  the distance between the locations of the vortex rings obtained in Theorem \ref{thm1} is of the order $ {\mathcal O} (\frac{1}{ |\log \epsilon|})$. Note that this distance is much smaller than the {\em leapfrogging} region in which the distance between the vortex rings is of the order $ {\mathcal O} (\frac{1}{ \sqrt{ |\log \epsilon|}})$. For the dynamics of vortex rings in the leapfrogging region for  the Gross-Pitaevskii equation we refer  to Jerrard-Smets \cite{Jerrard2016LeapfroggingVR} and the references therein.

The paper is organized as follows. In Section \ref{sec2}, we formulate the 3D problem as a two dimensional one. In Section \ref{sec3}, we introduce the approximate multi-vortex ring solutions and estimate their error. Section \ref{sec4} is devoted to the study of a nonlinear projected problem. This is more or less standard. The main part of the paper is Section \ref{sec5} and Section 6, where we get the reduced problem for the position of the vortex points and study some generating polynomials whose roots determine the location of the vortex rings.

\textbf{Acknowledgement }W. Ao is supported by NSFC no. 11631011, no. 11801421, and no. 12071357. Y. Liu is partially supported by NSFC no. 11971026 and \textquotedblleft
The Fundamental Research Funds for the Central Universities
WK3470000014\textquotedblright.  J. Wei is partially supported by NSERC of Canada.

\section{Formulation of the problem}\label{sec2}

We are looking for  a solution to problem (\ref{GrossPitaevskii}) in the form
$$
\Psi(y, t)
\,=\,
{\tilde u}\Big(y_1,y_2,y_3-c t\Big).
$$
Then $\tilde{u}$ must satisfy
\begin{equation}\label{GP-1}
-ic\frac{\partial \tilde{u}}{\partial y_3}=\Delta \tilde{u}+(1-|\tilde{u}|^2)\tilde{u}.
\end{equation}
Let $\varepsilon>0$ be a small parameter. We would like to seek solutions with traveling speed $c=\varepsilon|\ln \varepsilon|$.  Equation (\ref{TWK0}) then becomes
\begin{equation}
-i\varepsilon|\ln\varepsilon|\frac{\partial \tilde{u}}{\partial y_3}=\Delta \tilde{u}+(1-|\tilde{u}|^2)\tilde{u}\label{oure}.
\end{equation}
We require the solution $\tilde{u}$ satisfies
\begin{equation*}
\tilde{u}(y)\to 1\mbox{ as }|y|\to \infty.
\end{equation*}We are interested in the solutions axially symmetric with respect to the $y_3$ axis.
Let us introduce $x_1=\sqrt{y^2_1+y^2_2}, x_2=y_3,$ and
\begin{equation*}
z=x_1+ix_2,\, u(x_1,x_2)=\tilde{u}(y_1,y_2,y_3).
\end{equation*}
Then we get the following equation satisfied by $u$:
\begin{equation}\label{e1}
-i\varepsilon |\ln\varepsilon|\frac{\partial u}{\partial x_2}=\Delta_{(x_1,x_2)} u+\frac{1}{x_1}\frac{\partial u}{\partial x_1}+(1-|u|^2)u,
\end{equation}
with boundary conditions
\begin{equation*}
\frac{\p }{\p x_1}u(0,x_2)=0, \, u\to 1 \mbox{ as }|z|\to \infty.	
\end{equation*}

Observe that the problem (\ref{e1}) is invariant under the following two transformations:
\begin{equation*}
u(z)\to \overline{u(\bar{z})}, \, u(z)\to u(-\bar{z}).	
\end{equation*}
Thus we impose the following symmetry on the solutions $u$:
\begin{equation*}
\Sigma=\{u(z)= \overline{u(\bar{z})}, \, u(z)=u(-\bar{z}).	
\}	
\end{equation*}
This symmetry will play an important role in our analysis.
As a conclusion, if we write
$$
u(x_1,x_2)=u_1(x_1,x_2)+iu_2(x_1,x_2),
$$
then $u_1$ and $u_2$ enjoy the following conditions:
\begin{align}
\begin{aligned}\label{symmetryandboundary}
u_1(x_1,x_2)=u_1(-x_1,x_2),
&\qquad
u_1(x_1,x_2)=u_1(x_1,-x_2),
\\
u_2(x_1,x_2)=u_2(-x_1,x_2),
&\qquad
u_2(x_1,x_2)=-u_2(x_1,-x_2),
\\
\frac{\partial u_1}{\partial x_1}(0,x_2)=0,
&\qquad
\frac{\partial u_2}{\partial x_1}(0,x_2)=0.
\end{aligned}
\end{align}

We now have a two dimensional elliptic system with Neumann boundary condition $\frac{\p u}{\p x_1}(0,x_2)=0$. Compared with the two dimensional problem studied in \cite{Liu-Wei}, there are two differences: Firstly, there is an extra term $\frac{1}{x_1}\frac{\p u}{\p x_1}$; Secondly, the coefficient in front of $\frac{\p u}{\p x_2}$ becomes $\varepsilon |\ln \varepsilon|$, instead of $\varepsilon$.

Some remarks are in order. We aim to construct multi-vortex ring solutions to (\ref{oure}). For single vortex ring, one can use $v^+(x-p)v^-(x+p)$ as a good approximate solution for the equation (\ref{e1}). But for multi-vortex rings, the vortex-anti vortex pairs are not good enough because of the extra term $\frac{1}{x_1}\frac{\p u}{\p x_1}$ and the Neumann boundary condition. So we need to use more accurate approximate solution which we will explain in the next section.

\section{The approximate solution}\label{sec3}
In this section, we would like to define a family of approximate solutions for the equation (\ref{e1}).

\subsection{The first approximate solution}\label{subsec1}

We consider $\m K$ distinct points $p_j=(p_{j,1},p_{j,2})$, $j=1,...,\m K$, lying in the right half of the $z$ plane. Let us define $p_j^*=-\bar{p}_j$ for $j=1,\cdots,\m K$. We also denote $p_{\m K+j}=p_j^*$ for $j=1,\cdots,\m K$. Intuitively, these points represent the location of the vortex rings.
We also suppose that the set of points $\{p_1,...,p_{\m K}\}$ is symmetric with respect to the $x_1$ axis.  Moreover, we will assume:

\medskip

({\bf A1.})
\begin{eqnarray}\label{configuration}
\rho&:=min_{\ell \neq j, 1\leq \ell,j\leq \m K}\{|p_{\ell}-p_j|\}\sim \frac{1}{\varepsilon|\ln \varepsilon| },
\end{eqnarray}
and
$$ |p_{j,1}|\sim \frac{1}{\varepsilon }$$ for $j=1, \cdots, \m K.$

\medskip

n order to understand more clearly the difference between the 2D and 3D case, let us now following the strategy of \cite{Liu-Wei} to define an approximate solution.

Let $S=S_1$ be the function associated to the degree one vortex solution of the Ginzburg-Landau equation, defined in the first section. Define
\begin{equation*}
u_j:=S(|z-p_j|)e^{i\tau_j\theta_j}, \, \ j=1,\cdots,\m K,
\end{equation*}
where $\theta_j$ is the angle around $p_j$, $\tau_j=+1\mbox{ or} -1$, corresponding to  the degree $\pm 1$ vortex. We then set
\begin{equation*}
u_j=S(|z-p_j|)e^{-i\tau_{j-\m K}\theta_j}, \, \ j=\m K+1,\cdots,2\m K
\end{equation*}
where $\theta_j$ is the angle around $p_j$. The reason of defining these functions is the following: Projecting a vortex ring onto the $z$ plane, we get two circles in the right and left plane with different orientation. Here $u_j$ and $u_{j+\m K}$ can be viewed as a vortex-antivortex pair.

We now define the first approximate solution as
\begin{equation}\label{firstapproximation}
U=\Pi_{j=1}^{2\m K}u_j.
\end{equation}
We will see that this approximate solution is not good enough to handle the 3D case and later on we will introduce a refined approximate solution.  Note that at this moment, we still haven't decided the sign of the degree of the vortex. This will also be done later on.

Since each vortex in the right half plane has a vortex in the left plane with opposite sign, we can check directly that $U\to 1$ as $|z|\to \infty$.  We will see that the approximate solution satisfies the boundary and symmetry condition (\ref{symmetryandboundary}). In fact, by the choice of the vortex points, one has
\begin{lemma}\label{symmetry}
	The approximate solution $U$ has the following symmetry:
	\begin{equation*}
	U(\bar{z})=\bar{U}(z), \, \, U(z^*)=U(z)	
	\end{equation*}	
	where $z^*=-\bar{z}$.
\end{lemma}

\begin{proof}
	This is the result by the definition of the vortex points. Since $v_-(z)=\overline{v_+(z)}$, one has
	\begin{equation*}
	\begin{split}
	U(z^*)&=\Pi_{j=1}^{\m K} [u_j(z^*-p_j)u_{\m K+j}(z^*-p_j^*)]	\\
	&=\Pi_{j=1}^{\m K} [ u_j((z-p_j^*)^*)u_{\m K+j}((z-p_j)^*)]\\
	&=\Pi_{j=1}^{\m K} [u_{\m K+j}(z-p_j^*)u_j(z-p_j)]=U(z),
	\end{split}
	\end{equation*}
	and since the points $\{p_j\}$ are invariant with respect to the reflection across the $x_1$ axis, we have
	\begin{equation*}
	\begin{split}
	U(\bar{z})&=\Pi_{j=1}^{\m K} [ u_j(\bar{z}-p_j)u_{\m K+j}(\bar{z}-p_j^*)]	\\
	&=\Pi_{j=1}^{\m K}[ \bar{u}_j(z-\bar{p}_j)\bar{u}_{\m K+j}(z-\bar{p}_j^*)]=\bar{U}(z).
	\end{split}
	\end{equation*}
	This finishes the proof.

\end{proof}

\subsection{The error of the first approximate solution}\label{subsec2}
Firstly, we estimate the error of the first approximate solution $U$. Since it satisfies the symmetry and boundary condition (\ref{symmetryandboundary}), one only need to consider in the domain $\{x_1>0\}$.

Recall that the degree $\pm 1$ vortex satisfies
\begin{equation*}
S''(r)+\frac{1}{r}S'(r)-\frac{1}{r^2}S(r)+(1-S^2)S=0.
\end{equation*}
It has the following properties(\cite{Pacard}):
\begin{lemma}\label{propertyofvortex}

 The vortex solution satisfies the following properties:
\begin{itemize}	
\item[(i).]$S(0)=0, \, S'(r)>0, \, S(r)\in (0, 1) $;
\item[(ii.)]
$S(r)=1-\frac{1}{2r^2}+O(\frac{1}{r^4})$ as $r\to \infty$;	
\item[(iii).]$S(r)= a_0r-\frac{a_0}{8}r^3+O(r^5)$ as $r\to 0$ where $a_0$ is a positive constant.
\end{itemize}
\end{lemma}

In this subsection, we are going to estimate the error caused by the first approximation $U$. Use $E_1$ to denote the error :
\begin{eqnarray*}
E_1=i\ve|\ln \ve|\frac{\partial U}{\partial x_2}+\Delta U+(1-|U|^2)U+\frac{1}{x_1}\frac{\partial U}{\partial x_1}.
\end{eqnarray*}

We have
\begin{equation*}
\begin{split}
\Delta U&=\Delta(u_1\cdots u_{2\m K})\\
&=\sum_{\ell=1}^{\m K}(\Delta u_\ell\Pi_{j\neq \ell}u_j)+\sum_{j\neq \ell}\nabla u_{\ell}\cdot \nabla u_j\Pi_{t\neq \ell,j}u_{t}.	
\end{split}	
\end{equation*}
On the other hand, writing $\rho_j=|u_j|^2-1$, one has
\begin{equation*}
\begin{split}
|U|^2-1	&=\Pi_{j=1}^{2\m K}(1+\rho_j)-1\\
&=\sum_j \rho_j+\sum_{j=1}^{2\m K}\mathcal{Q}_j,
\end{split}
\end{equation*}
where $\mathcal{Q}_j=\sum_{i_1<i_2<\cdots<i_i}(\rho_{i_1}\cdots\rho_{i_i})$. Using the fact that $\Delta u_j-\rho_j u_j=0$, we get
\begin{equation*}
\begin{split}
\Delta U+(1-|U|^2)U&=\sum_{\ell \neq j}(\nabla u_{\ell}\cdot \nabla u_j\Pi_{t\neq \ell,j}u_{t})-U\sum_{j=2}^{2\m K}\mathcal{Q}_j.	
\end{split}
\end{equation*}
Let
$$\varphi_0=\sum_{j=1}^{\m K}{\tau_j}(\theta_j-\theta_{\m K +j}),$$
and
$$ \, r_j=|z-p_j|, \, r_{\m K+j}=|z-q_j|, r_{j,1}=x_1-p_{j,1}, r_{j,2}=x_2-p_{j,2}.$$ We have
\begin{equation*}
\begin{split}
\frac{1}{x_1}\frac{\partial U}{\partial x_1}&=\frac{1}{x_1}\Big[e^{i\varphi_0}\frac{\partial }{\partial x_1}(\Pi_{j=1}^{\m K} S(r_j)S(r_{\m K+j}))+ie^{i\varphi_0}\frac{\partial \varphi_0}{\partial x_1}\Pi_{j=1}^{\m K} S(r_j)S(r_{\m K+j})\Big]	\\
&=\Big(\sum_{j=1}^{2\m K}\frac{1}{x_1}\frac{S'(r_j)}{S(r_j)}\frac{r_{j,1}}{r_j}+\frac{i}{x_1}\frac{\partial \varphi_0}{\partial x_1}\Big)U.
\end{split}	
\end{equation*}
Similarly, there holds
\begin{equation*}
\begin{split}
\frac{\partial U}{\partial x_2}&=\sum_{j=1}^{2\m K}\partial_{x_2}u_j\Pi_{\ell\neq j}u_{\ell}\\
&=\Big(\sum_{j=1}^{2\m K}\frac{S'(r_j)}{S(r_j)}\frac{r_{j,2}}{r_j}+i\frac{\partial \varphi_0}{\partial x_2}	\Big)U.
\end{split}	
\end{equation*}	
Combining the above computations, we obtain
\begin{equation*}
\begin{split}
E_1&=\sum_{\ell\neq j}\frac{(\nabla u_{\ell}\cdot \nabla u_j)}{u_{\ell} \, u_j}U-U\sum_{j=2}^{2\m K}\mathcal{Q}_j\\
&	+
\Big(\sum_{j=1}^{2\m K}\frac{1}{x_1}\frac{S'(r_j)}{S(r_j)}\frac{r_{j,1}}{r_j}+\frac{i}{x_1}\frac{\partial \varphi_0}{\partial x_1}\Big)U\\
&+i\ve|\ln\ve|\Big(\sum_{j=1}^{2\m K}\frac{S'(r_j)}{S(r_j)}\frac{r_{j,2}}{r_j}+i\frac{\partial \varphi_0}{\partial x_2}	\Big)U.
\end{split}	
\end{equation*}
In the sequel, we denote $|p_{j,1}|$ by $d_j$. Direct computation yields

\begin{equation*}
\begin{split}
\frac{\partial \varphi_0}{\partial x_2}&=\sum_{j=1}^{\m K}\tau_j\big[ \frac{\partial \theta_j}{\partial x_2}-\frac{\partial \theta_{\m K+j}}{\partial x_2} \big]\\
&=\sum_{j=1}^{\m K}\tau_j\Big[\frac{x_1-p_{j,1}}{r^2_j}-\frac{x_1-p^*_{j,1}}{r^2_{\m K+j}}\Big]\\
&=\sum_{j=1}^{\m K}\tau_j\frac{2d_j(x_1^2-p_{j,1}^2-(x_2-p_{j,2})^2)}{r_j^2r_{\m K+j}^2}.	
\end{split}
\end{equation*}
We also have
\begin{equation*}
\begin{split}
\frac{1}{x_1}\frac{\partial \varphi_0}{\partial x_1}&=\frac{1}{x_1}\sum_{j=1}^{\m K}\tau_j\Big[\frac{\partial \theta_j}{\partial x_1}-\frac{\partial \theta_{\m K+j}}{\partial x_1}\Big]\\
&=-\frac{1}{x_1}\sum_{j=1}^{\m K}\tau_j\Big[\frac{x_2-p_{j,2}}{r^2_j}-\frac{x_2-p^*_{j,2}}{r^2_{\m K+j}}\Big].
\end{split}	
\end{equation*}
Observe that $\frac{1}{x_1}\frac{\partial \varphi_0}{\partial x_1}$ contributes to the imaginary part of the error $E_1$. Note that away from the vortex point $p_{j}$, this decays only at the rate $O(r_j^2)$, which is not sufficient for our construction. Hence the vortex-antivortex pair is not enough to be a good approximate solution.

\subsection{The reference vortex ring}\label{subsec3}

In order the get rid of these singularities, one need more accurate approximations for the vortex ring.

In \cite{Jerrard2016LeapfroggingVR}, leap frogging behavior of the vortex rings to the GP equation has been analyzed. Indeed, our construction in this paper is partly inspired by these leap frogging behavior. Following the analysis performed in \cite{Jerrard2016LeapfroggingVR}, we introduce the potential function $A_a$, which satisfies the following equation:

\begin{equation}\label{A}
\left\{
\begin{array}{l}
-div\Big(\frac{1}{x_1}\nabla (x_1A_a) \Big)=2\pi \delta_a \mbox{ in }H,\\
A_a=0 \mbox{ on }\p H	,
\end{array}
\right.	
\end{equation}
where $H=\{(x_1, x_2)\in \R^2, \, x_1>0\}$ and $a\in H$.

For the region $\{x_1\leq 0\}$, we consider the odd extension of $A_a$. The expression of $A_a$ can be integrated explicitly in terms of complete elliptic integrals (see \cite{Jackson,Jerrard2016LeapfroggingVR}). We emphasize that in the literature, there are different notations concerning the definition of complete elliptic integrals, mainly about its arguments.

Let $r:=r_a=|z-a|$. When $r_a=o(|a_1|)$, one has the following asymptotic behavior
\begin{equation}\label{asymp1}
A_a(z)=\Big( \ln \frac{a_1}{r_a}+3\ln 2-2 \Big)+O\Big( \frac{r_a}{a_1}|\ln \frac{r_a}{a_1}| \Big)
\end{equation}
and
\begin{equation}\label{asymp2}
\p_r A_a=-\frac{1}{r}+O(\frac{1}{a_1}), 	
\end{equation}
and for $x_1\to 0$
\begin{equation}\label{asymp3}
A_a(x_1,x_2)=\frac{x_1a_1^2}{a_1^3+x_2^2}\mbox{ as }\frac{x_1}{a_1}\to 0.
\end{equation}

Up to a constant phase factor, there exists a unique unimodular map $u_a^*\in C^\infty(H\setminus\{a\}, S^1)$ such that
\begin{equation}\label{relation}
x_1(iu_a^*, \nabla u_a^*)=x_1j(u_a^*)=-\nabla^\perp(x_1A_a),	
\end{equation}
where
$$
j(u)=u\times \nabla u=(iu, \nabla u)=Re(iu \nabla \bar{u}).
$$
In the sense of distribution, we have
\begin{equation*}
\left\{
\begin{array}{l}
div(x_1j(u_a^*))=0, \\
 curl(j(u_a^*))=2\pi \delta_a,	
 \end{array}
 \right.
\end{equation*}
and the function $u_a^*$ corresponds to a singular vortex ring centered at $a$.

If we denote by $u_a^*=e^{i\varphi_a}$, then by (\ref{relation}), one has
\begin{equation}\label{relation1}
\p_1\varphi_a=\p_2 A_a, \, \p_2\varphi_a=-\frac{1}{x_1}\frac{\p (x_1A_a)}{\p x_1}	.
\end{equation}
So from the definition of $\varphi_a$ and the boundary condition of $A_a$, one has

\begin{equation*}
\left\{
\begin{array}{l}
\Delta \varphi_a+\frac{1}{x_1}\frac{\p \varphi_a}{\p x_1}=0 \mbox{ in } H, \\
\p_1 \varphi_a(0, x_2)=0 \mbox{ on }\p H	.
\end{array}
\right.	
\end{equation*}
Moreover, using the relation of $\varphi_a$ and $A_a$ in (\ref{relation1}), one has
\begin{equation}\label{asymp4}
	\nabla \varphi_a=\frac{1}{r_a}\nabla^\perp r_a+O(\frac{1}{x_1}\log \frac{a_1}{r_a}) \mbox{ for }\frac{r_a}{a_1}\to 0
\end{equation}
and
\begin{equation}\label{asymp5}
	|\nabla \varphi_a|\leq \frac{C}{1+r_a}\mbox{ for }r_a\geq 1.
\end{equation}
 So near the vortex point $a$, $\nabla \varphi_a$ can be viewed as a perturbation of $\nabla \theta_a$.

\subsection{Improvement of the first approximate solution}\label{subsec4}
We will use $u_a^*$ instead of the vortex-anti vortex pair $e^{i(\theta_j-\theta_{\m K +j})}$ to define a more accurate approximate vortex ring. In view of the symmetry condition (\ref{symmetryandboundary}), the vortex ring associated to a point $a\in H$ will defined to be
\begin{equation*}
S(|z-a|)S(|z-a^*|)u_a^*(x)=S(|z-a|)S(|z-a^*|)e^{i\varphi_a(z)}.
\end{equation*}
We can also decompose $\varphi_a$ as
\begin{equation*}
\varphi_a(z)=\theta_a(z)-\theta_{a^*}(z)+\tilde{\varphi}_a.	
\end{equation*}
Note that the difference $\tilde{\varphi}_a$ can be analyzed around the vortex point $a$ using the asymptotic behavior of $A$.

Define
\begin{equation*}
\varphi_d=\sum_{j=1}^{\m K}\tau_j\varphi_{p_j}=\varphi_0+\sum_{j=1}^{\m K}\tau_j\tilde{\varphi}_{p_j}:=\varphi_0+\tilde{\varphi}_{\bf p}	.
\end{equation*}
Then our final approximation will be defined as
\begin{equation*}
	\m U(x)=U(x)e^{i\tilde{\varphi}_{\bf p}}=\Pi_{j=1}^{2\m K}S_{p_j}(x)e^{i\varphi_d}.
\end{equation*}
Namely, we replace the function $\varphi_0$ by $\varphi_{d}$ in the first approximate solution.
Since $\{p_j\}$ satisfies ({\bf A1}),  one can see that the new approximate solution will satisfy the symmetry condition (\ref{symmetryandboundary}).

\medskip

\subsection{Error of the final approximation}\label{subsec5}
Now the new error becomes
\begin{equation*}
\begin{split}
E_2&=i\ve|\ln \ve|\frac{\partial \mathcal{U}}{\partial x_2}+\Delta \mathcal{U}+(1-|\mathcal{U}|^2)\mathcal{U}+\frac{1}{x_1}\frac{\partial \mathcal{U}}{\partial x_1}\\
&:=E_{21}+E_{22}	.
\end{split}
\end{equation*}
Here $E_{21}$ is the first term in the left hand side. We use $B_l(p)$ to denote the ball of radius $l$ centered at the point $p$.
We have the following error estimate:
\begin{lemma}\label{lemmaforerror}
There exists a constant $C$ such that for all small $\ve $ and all points $p_j$ satisfying ({\bf A1}), we have
\begin{equation*}
\sum_{j=1}^{2\m K}\|E_2\|_{L^{9}(B_3(p_j))}\leq C\ve|\ln \ve|.
\end{equation*}	
Moreover, we have $E_2=i\mathcal{U}[R_1+iR_2]$, with $R_1, \, R_2$ real valued and
\begin{equation*}
\begin{split}
|R_1|
&\leq C\sum_{j=1}^{2\m K}\frac{O(\ve^{1-\delta})}{(1+r_j)^3},\\
|R_2|&\leq C\sum_{j=1}^{2\m{ K}}\frac{O(\ve^{1-\delta})}{1+r_j},
\end{split}	
\end{equation*}
for any $\delta \in(0,1)$, if $|z-p_j|>1$ for all $j$.

\end{lemma}

\begin{proof}

We compute, in $B_{\frac{\rho}{5}}(p_j)$,
\begin{equation}\label{estimateofe21}
\begin{split}
\frac{\partial \mathcal U}{\partial x_2}&=\frac{\p \Pi_{j=1}^{2\m K}S_j}{\p x_2}e^{i\varphi_d}+i\mathcal{U}\nabla \varphi_d\\
&=\big[\sum_{j=1}^{2\m K}\frac{S'(r_j)}{S(r_j)}\frac{r_{j,2}}{r_j}+i \nabla \varphi_d\big]\mathcal{U}.
\end{split}	
\end{equation}
Hence in $(\cup_{j=1}^{2\m K} B_{3}(p_j))^c$, by (\ref{asymp5}),  we have
\begin{equation*}
\begin{split}
Re\Big[\frac{E_{21}}{i\mathcal{U}}\Big]&=\ve |\ln \ve|	\big[\sum_{j=1}^{2\m K}\frac{S'(r_j)}{S(r_j)}\frac{r_{j,2}}{r_j}\big]\\
&\leq C\sum_{j=1}^{2\m K}\frac{O(\ve^{1-\delta})}{(1+r_j)^3},\\
Im\Big[\frac{E_{21}}{i\mathcal{U}}\Big]&=\ve |\ln \ve|	[\nabla \varphi_d]\\
&\leq C\sum_{j=1}^{2\m K}\frac{O(\ve^{1-\delta})}{(1+r_j)}.
\end{split}	
\end{equation*}
We also have
\begin{equation*}
\|i\ve |\ln \ve|\,\partial_{x_2}\mathcal{U}\|_{L^9(\cup_j \{r_j\leq 3\})}\leq C\ve|\ln \ve|.	
\end{equation*}
Note that the $L^\infty$ norm is not bounded near $p_j$, due to the presence of $\ln r_j$ term.

Next, letting $S_j=S(r_j)$ and using the fact that
\begin{equation*}
\Delta S_j-\frac{S_j}{r_j^2}+(1-S_j^2)S_j=0,
\end{equation*}
one has
\begin{equation}\label{estimateofe22}
\begin{split}
E_{22}&=\Delta \mathcal{U}+(1-|\mathcal{U}|^2)\mathcal{U}+\frac{1}{x_1}\frac{\partial \mathcal{U}}{\partial x_1}\\
&=\mathcal{U}\Big[ \sum_{j=1}^{2\m K}\frac{1}{r_j^2}-|\nabla \varphi_d|^2+\frac{1}{x_1}\sum_{j=1}^{2\m K}\frac{S'(r_j)}{S(r_j)}\p_{x_1}r_j-\sum_{j=1}^{2\m K} \mathcal{Q}_j \\ &+2i\sum_{j=1}^{2\m K} \frac{S'(r_j)}{S(r_j)}\nabla r_j\cdot\nabla \varphi_d \Big]
\end{split}	
\end{equation}
where we have used the fact that
\begin{equation*}
\Delta \varphi_d+\frac{1}{x_1}\frac{\partial }{\partial x_1}\varphi_d=0.	
\end{equation*}
By carefully checking the terms, using (\ref{asymp1})-(\ref{asymp5}), away from the vortex points, one has

\begin{equation*}
\begin{split}
Re \Big[\frac{E_{22}}{i\mathcal{U}}\Big]
&\leq C\sum_{j=1}^{2\m K}\frac{O(\ve^{1-\delta})}{(1+r_j)^3},\\
Im\Big[\frac{E_{22}}{i\mathcal{U}}\Big]
&\leq C\sum_{j=1}^{2\m K}\frac{O(\ve^{1-\delta})}{(1+r_j)}.
\end{split}	
\end{equation*}
Moreover,
\begin{equation*}
\|E_{22}\|_{L^9(\cup_j\{ r_j\leq 3\})}\leq C\ve|\ln \ve|.	
\end{equation*}
Combining the estimates for $E_{21}$ and $E_{22}$, we obtain the desired estimates.

\end{proof}


\section{Linear theory}\label{sec4}

Now we set up the reduction procedure. The linear theory we use here will be the same one as that of \cite{Liu-Wei}. We recall the framework developed there in the sequel. As usual, we shall look for a solution of (\ref{e1}) in the form:

\begin{equation}
u:=\left(  \m U+ \m U\eta\right)  \chi+\left(  1-\chi\right)  \m U e^{\eta}, \label{pertubation}%
\end{equation}
where $\chi$ is a cutoff function such that
\begin{equation*}
\chi(x)=\sum_{j=1}^{2\m K}\tilde{\chi}(x-p_j)
\end{equation*}
and $\tilde{\chi}(s)=1$ for $s\leq 1$ and $\tilde{\chi}(s)=0$ for $s\geq 2$ and
$\eta=\eta_{1}+\eta_{2}i$ is complex valued function close to $0$ in
suitable norm which will be introduced below. We also assume that $\eta$ has the
same symmetry as $\m U.$ Note that near the vortice, $u$ is obtained from $\m U$ by an additive perturbation; while away from the vortice, $u$ is of the form $\m Ue^{\eta}$. The reason of choosing the perturbation $\eta$ in the form (\ref{pertubation}) is explained in Section 3 of \cite{del2006variational} and also in \cite{Liu-Wei}. Essentially, the form of the perturbation far away from the origin makes it easier to handle the decay rates of the error away from the origin.

\medskip
The conditions imposed on $u$ in (\ref{symmetryandboundary}) can be transmitted to $\eta$:
\begin{align}\label{bdyofpsi}
\begin{aligned}
\eta_1(x_1,x_2)=\eta_1(-x_1,x_2),
&\qquad
\eta_1(x_1,x_2)=-\eta_1(x_1,-x_2),
\\
\eta_2(x_1,x_2)=\eta_2(-x_1,x_2),
&\qquad
\eta_2(x_1,x_2)=\eta_2(x_1,-x_2),
\\
\frac{\partial \eta_1}{\partial x_1}(0,x_2)=0,
&\qquad
\frac{\partial \eta_2}{\partial x_1}(0,x_2)=0.
\end{aligned}
\end{align}

In view of (\ref{pertubation}), we can write
$u=\m Ue^{\eta}+\gamma,$ where
\[
\gamma:=\chi \m U\left(  1+\eta-e^{\eta}\right)  .
\]
Note that $\gamma$ is localized near the vortex points and of the order $o(\eta),$ for $\eta$ small.

Set $\mathcal{A}:=\left(  \chi+\left(  1-\chi\right)  e^{\eta}\right)  \m U.$ Then $u$ can  be written as
$$u=\m U\eta\chi+\mathcal{A}.$$
Following the computation in \cite{Liu-Wei}, we get
\[
\left(  1-\left\vert u\right\vert ^{2}\right) u =\left(  \m U\eta\chi+\mathcal{A}\right)
\left(  1-\left\vert \m U e^{\eta}+\gamma\right\vert ^{2}\right)  .
\]

The equation for $\eta$ becomes
\begin{equation}
-\mathcal{A}\mathbb{L}\left(  \eta\right)  =\left(  1+\eta\right)  \chi E_2\left(
\m U\right)  +\left(  1-\chi\right)  e^{\eta}E_2\left(  \ U\right)  +N_{0}\left(
\eta\right)  , \label{GPF}%
\end{equation}
where $E_2\left(  \m U\right)  $ represents the error of the approximate solution $\m U$, and
\begin{equation}
\label{Lcali}
\mathbb{L}\eta:=i\varepsilon\frac{\p\eta}{\p x_2}+\Delta\eta+2u^{-1}\nabla
u\cdot\nabla\eta-2\left\vert u\right\vert ^{2}\eta_{1}+\frac{1}{x_1}\frac{\p \eta}{\p x_1},
\end{equation}
while $N_{0}$ is $o(\eta),$ and explicitly given by%
\begin{align*}
 N_{0}\left(  \eta\right)  &:=\left(  1-\chi\right)  \m Ue^{\eta}\left\vert
\nabla\eta\right\vert ^{2}+i\varepsilon|\ln \ve|\left(  \m U\left(  1+\eta-e^{\eta
}\right)  \right)  \partial_{x_2}\chi\\
&+\frac{1}{x_1}\left(  \m U\left(  1+\eta-e^{\eta
}\right)  \right)  \partial_{x_1}\chi+2\nabla\left(  \m U\left(  1+\eta-e^{\eta}\right)  \right)  \cdot\nabla
\chi+\m U\left(  1+\eta-e^{\eta}\right)  \Delta\chi\\
&  -2\m U\left\vert \m U\right\vert ^{2}\eta\eta_{1}\chi-\left(  \mathcal{A}+\m U\eta\chi\right)
\left[  \left\vert \m U\right\vert ^{2}\left(  e^{2\eta_{1}}-1-2\eta_{1}\right)
+\left\vert \gamma\right\vert ^{2}+2\operatorname{Re}\left(  \m Ue^{\eta}%
\bar{\gamma}\right)  \right]  .
\end{align*}

Let us write this equation as
\begin{equation}\label{realequation}
\mathbb{L}\left(  \eta\right)  =-{\m U}^{-1}E_2\left(  \m U\right)  +N\left(
\eta\right)  ,
\end{equation}
where
\begin{equation*}
\begin{split}
N(\eta)&=-\left\vert \m U\right\vert ^{2}\left(  e^{2\eta
_{1}}-1-2\eta_{1}\right)  +\left\vert \nabla\eta\right\vert ^{2}\\
&  +i\varepsilon|\ln \ve| \mathcal{A}^{-1}\left(  \m U\left(  1+\eta-e^{\eta}\right)  \right)
\partial_{x_2}\chi+\frac{1}{\mathcal{A}x_1}\left(  \m U\left(  1+\eta-e^{\eta
}\right)  \right)  \partial_{x_1}\chi\\
&+2\mathcal{A}^{-1}\nabla\left(  \m U\left(  1+\eta-e^{\eta}\right)
\right)  \cdot\nabla\chi\\
&  +\mathcal{A}^{-1}\m U\left(  1+\eta-e^{\eta}\right)  \Delta\chi-\mathcal{A}^{-1}\m U\chi\left\vert
\nabla\eta\right\vert ^{2}-\left\vert \gamma\right\vert ^{2}%
-2\operatorname{Re}\left(  \m Ue^{\eta}\bar{\gamma}\right) \\
&  +\mathcal{A}^{-1}\m U\eta\chi\left[  {\m U}^{-1}E_2\left(  \m U\right)  -2\left\vert \m U\right\vert
^{2}\eta_{1}-\left\vert \m U\right\vert ^{2}\left(  e^{2\eta_{1}}-1-2\eta
_{1}\right)  -\left\vert \gamma\right\vert ^{2}-2\operatorname{Re}\left(
\m Ue^{\eta}\bar{\gamma}\right)  \right]  .
\end{split}
\end{equation*}
This nonlinear equation, equivalent to the original GP equation,  is the one we eventually want to solve. Observe that in $N\left(  \eta\right)  $, except $\left\vert \m U\right\vert
^{2}\left(  e^{2\eta_{1}}-1-2\eta_{1}\right)  -\left\vert \nabla
\eta\right\vert ^{2},$ other terms are all localized near the vortex points.

\subsection{A Linear problem}
By the definition of our vortex configuration, one can see that the terms contain $\ve |\ln \ve|$ and $\frac{1}{x_1}$ can be viewed as small perturbation near the vortex points.

Let us first consider the following linear problem:
\begin{equation}\label{linear}
\mathbb{L}(\eta)=h, \, Re\int_{\R^2}\overline{\m U \eta }Z_{\ell j}=0, \, \eta \, \mbox{ satisfies }\, (\ref{bdyofpsi}),
\end{equation}

where
\begin{equation*}
Z_{\ell j}=\alpha_\ell\nabla u_\ell\, \tilde{\rho}_\ell(x), \alpha_\ell=\frac{\m U}{u_\ell},
\end{equation*}
and $\tilde{\rho}_\ell$ is a cutoff function centered at $p_\ell$ with support in $B_{\frac{\rho}{5}}(p_\ell)$.
We shall establish a priori estimates for this problem.
The following weighted norms and linear theory has been studied in \cite{Liu-Wei}.

Recall that $r_{j},j=1,\cdot\cdot\cdot,\m K,$ represent
the distance to the $j$-th vortex point. Let $w$ be a weight function defined
by
\[
w(z):=\left(  \sum_{j=1}^{2\m K}\left(  1+r_{j}\right)  ^{-1}\right)  ^{-1}.
\]
This function measures the minimal distance from the point $z$ to those vortex points. We
use $B_{a}\left(  z\right)  $ to denote the ball of radius $a$ centered at
$z.$ Let $\alpha,\sigma\in\left(  0,1\right)  $ be small positive numbers. For complex valued function
$\eta=\eta_{1}+\eta_{2}i,$ we define the following weighted
norm:
\begin{align*}
&  \left\Vert \eta\right\Vert _{\ast}\\
&  =\left\Vert u\eta\right\Vert _{W^{2,9}\left(  w<3\right)  }+\left\Vert
w^{1+\sigma}\eta_{1}\right\Vert _{L^{\infty}\left(  w>2\right)  }+\left\Vert
w^{2+\sigma}(|\nabla\eta_{1}|+|\nabla^{2}\eta_{1}|)\right\Vert _{L^{\infty
}\left(  w>2\right)  }\\
&  +\sup_{z\in\left\{  w>2\right\}  }\sup_{z_{1},z_{2}\in B_{w/3}\left(
z\right)  }\left(  \frac{\left\vert
\nabla\eta_{1}\left(  z_{1}\right)  -\nabla\eta_{1}\left(  z_{2}\right)
\right\vert +\left\vert \nabla^{2}\eta_{1}\left(  z_{1}\right)  -\nabla
^{2}\eta_{1}\left(  z_{2}\right)  \right\vert }{w\left(  z\right)  ^{-2-\sigma-\alpha}\left\vert z_{1}%
-z_{2}\right\vert ^{\alpha}}\right) \\
&  +\left\Vert w^{\sigma}\eta_{2}\right\Vert _{L^{\infty}\left(  w>2\right)
}+\left\Vert w^{1+\sigma}\nabla\eta_{2}\right\Vert _{L^{\infty}\left(
w>2\right)  }+\left\Vert w^{2+\sigma}\nabla^{2}\eta_{2}\right\Vert
_{L^{\infty}\left(  w>2\right)  }\\
&  +\sup_{z\in\left\{  w>2\right\}  }\sup_{z_{1},z_{2}\in B_{w/3}\left(
z\right)  }\left(  w\left(  z\right)  ^{1+\sigma+\alpha}\frac{\left\vert
\nabla\eta_{2}\left(  z_{1}\right)  -\nabla\eta_{2}\left(  z_{2}\right)
\right\vert }{\left\vert z_{1}-z_{2}\right\vert ^{\alpha}}\right) \\
&  +\sup_{z\in\left\{  w>2\right\}  }\sup_{z_{1},z_{2}\in B_{w/3}\left(
z\right)  }\left(  w\left(  z\right)  ^{2+\sigma+\alpha}\frac{\left\vert
\nabla^{2}\eta_{2}\left(  z_{1}\right)  -\nabla^{2}\eta_{2}\left(
z_{2}\right)  \right\vert }{\left\vert z_{1}-z_{2}\right\vert ^{\alpha}
}\right).
\end{align*}
Basically, the norm means that the real part of $\eta$ decays like $w^{-1-\sigma}$ and its first and second derivatives decay like $w^{-2-\sigma}$. Moreover, the imaginary part of $\eta$ only decays as $w^{-\sigma}$, but its first and second derivative decay as $w^{-1-\sigma}$ and  $w^{-2-\sigma}$ respectively. It is worth mentioning that the H\"{o}lder norms are taken into account in the definition because eventually we shall use the Schauder estimates. Moreover, near the vortex points, we use the $L^p$ norm, because the $L^\infty$ norm is not bounded there.

On the other hand, for complex valued function $h=h_{1}+ih_{2},$ we define the following weighted H\"{o}lder norm
\begin{align*}
\left\Vert h\right\Vert _{\ast\ast}  &  :=\left\Vert uh\right\Vert
_{L^{9}\left(  w<3\right)  }+\left\Vert w^{1+\sigma}h_{1}\right\Vert
_{L^{\infty}\left(  w>2\right)  }\\
&  +\left\Vert w^{2+\sigma}\nabla h_{1}\right\Vert _{L^{\infty}\left(
w>2\right)  }+\left\Vert w^{2+\sigma}h_{2}\right\Vert _{L^{\infty}\left(
w>2\right)  }\\
&  +\sup_{z\in\left\{  w>2\right\}  }\sup_{z_{1},z_{2}\in B_{w/3}\left(
z\right)  }\left(  w\left(  z\right)  ^{2+\sigma+\alpha}\frac{\left\vert
\nabla h_{1}\left(  z_{1}\right)  -\nabla h_{1}\left(  z_{2}\right)
\right\vert }{\left\vert z_{1}-z_{2}\right\vert ^{\alpha}}\right) \\
&  +\sup_{z\in\left\{  w>2\right\}  }\sup_{z_{1},z_{2}\in B_{w/3}\left(
z\right)  }\left(  w\left(  z\right)  ^{2+\sigma+\alpha}\frac{\left\vert
h_{2}\left(  z_{1}\right)  -h_{2}\left(  z_{2}\right)  \right\vert
}{\left\vert z_{1}-z_{2}\right\vert ^{\alpha}}\right)  .
\end{align*}
This definition tells us that the real and imaginary parts of $h$ have different decay rates. Moreover, intuitively we require $h_{1}$ to gain one more power of decay at
infinity after taking one derivative. The choice of this norm is partly decided by the decay and smooth properties of $E_2(\m U)$.

We have the following a priori estimate for solutions of the equation (\ref{linear}).

\begin{lemma}[Proposition 4.5 in \cite{Liu-Wei}]
Let $\varepsilon>0$ be small. Suppose $\eta$ is a solution of (\ref{linear}) with $\left\Vert h\right\Vert _{\ast\ast}< \infty$. Then
$$\left\Vert \eta\right\Vert _{\ast}\leq C\varepsilon^{-\sigma}\left\vert
\ln\varepsilon\right\vert \left\Vert h\right\Vert _{\ast\ast}$$
 where $C$ is a constant independent of $\varepsilon$ and $h$.
\end{lemma}


We now consider the following linear projected problem:
\begin{equation}\label{projectedproblem}
\left\{\begin{array}{l}
	\mathbb{ L}(\eta)=h+\sum_{j=1}^{2\m K}\sum_{j=1}^2 c_{\ell j}Z_{\ell j}, \\ Re\int_{\R^2}\overline{{\m U} \eta }Z_{\ell j}\, dx=0, \\
\eta \mbox{ satisfies }(\ref{bdyofpsi}).
\end{array}
\right.
\end{equation}
We state the following existence result:
\begin{proposition}\label{pro-of-existence}
There exists constant $C$, depending only on $\alpha, \, \sigma$ such that for all $\ve  $ small, the following holds: if $\|h\|_{**}<\infty$, there exists a unique solution $(\eta, \{c_{\ell j}\})=T_\ve(h)$ to (\ref{projectedproblem}). Furthermore, there holds
\begin{equation*}
\|\eta\|_*\leq C\ve^{-\sigma}|\ln\ve|\|h\|_{**}	.
\end{equation*}	
\end{proposition}
\begin{proof}
The proof is similar to that of Proposition 4.1 in \cite{del2006variational}. Instead of solving (\ref{projectedproblem}) in $\R^2$, we solve it in a bounded domain first:
\begin{equation*}
\left\{\begin{array}{l}
\mathbb{ L}(\eta)=h+\sum_{\ell=1}^{2\m K}\sum_{j=1}^2 c_{\ell j}Z_{\ell j}, \, Re\int_{B_M}\overline{\m U\eta}Z_{\ell j}\, dx=0 \mbox{ in } B_M(0)\\
\eta=0\mbox{ on }\p B_M(0), \\
 \eta \mbox{ satisfies \, the \, condition \, (\ref{bdyofpsi})},
\end{array}
\right.
\end{equation*}
where $M$ large enough. By the standard proof of a priori estimates, we also obtain the following estimates for any solution $\eta_M$ of the above problem with
\begin{equation*}
\|\eta_M\|_*\leq C\ve^{-\sigma}|\ln\ve|\|h\|_{**}. 	
\end{equation*}
By working in the Sobolev space $H_0^1(B_M)$, the existence will follow by Fredholm alternatives. Now letting $M\to \infty$, we obtain a solution of the required properties.
	
\end{proof}

\subsection{Projected nonlinear problem}\label{subsec7}
From now on, we will denote by $T_\ve(h)$ the solution of (\ref{projectedproblem}). We consider the following nonlinear projected problem :

\begin{equation}\label{nonlinearproblem}
\left\{\begin{array}{l}
\mathbb{ L}(\eta)+\frac{E_2(\m U)}{\m U}+N(\eta)=\sum_{\ell=1}^{2\m K}\sum_{j=1}^2 c_{\ell j}Z_{\ell j}, \, \\
 Re\int_{\R^2} \ov{\m U \eta }Z_{\ell j}\, dx=0 , \\
 \eta \mbox{ satisfies \, the \, condition }\, (\ref{bdyofpsi}).
\end{array}
\right.		
\end{equation}
Using the operator $T_\ve  $ defined in Proposition \ref{pro-of-existence}, we can write the above problem as
\begin{equation*}
\eta=T_\ve(\frac{E_2(\m U)}{\m U}-N(\eta)):=G_\ve(\eta). 	
\end{equation*}
Using the error estimates in Lemma \ref{lemmaforerror}, we have for $r_j\sim \ve^{-1}$,
\begin{equation*}
Re(\frac{E_2}{\m U})\sim \frac{\ve^{1-\delta}}{r_j}, \,
Im(\frac{E_2}{\m U})\sim \frac{\ve^{1-\delta}}{r_j^3}.
\end{equation*}
More precisely, if one check the express of the error, and using the explicate expression for $A_a$ in Section \ref{sec2}, one can check by direct calculation that for $r_j>>\ve^{-1}$,
\begin{equation*}
Re(\frac{E_2}{\m U})\sim \frac{\ve^{1-\delta}}{r_j^2}, \,
Im(\frac{E_2}{\m U})\sim \frac{\ve^{1-\delta}}{r_j^3}.
\end{equation*}
Taking this into account, one has
\begin{equation*}
\|{\m U}^{-1}E_2(\m U)\|_{**}\leq C\ve^{1-\delta}
\end{equation*}
for any $\delta>0$.

Let
\begin{equation*}
\eta\in B:=\{\|\eta\|_*\leq C\ve^{1-\beta}\}	
\end{equation*}
for $\beta\in (\delta+\sigma,1)$.
Then using the explicit form of $N(\eta)$, we have

\begin{equation*}
\|G_\ve(\eta)\|_{*}\leq C(\| N(\eta)\|_{**}+\|{\m U}^{-1}E_2(\m U)\|_{**})\leq C\ve^{1-\beta}	
\end{equation*}
and
\begin{equation*}
\|G_\ve(\eta)-G_\ve(\tilde{\eta})\|_{*}\leq o(1)\|\eta-\tilde{\eta}\|_*
\end{equation*}
for all $\eta, \, \tilde{\eta}\in B$.
By contraction mapping theorem, we obtain the following:
\begin{proposition}
	There exists constant $C$ and $\beta$ small, depending only on $\alpha, \, \sigma$ such that for all $\ve  $ small, the following holds:  there exists a unique solution $(\eta_{\ve, \{p_i\}}, \{c_{ij}\})=T_\ve(h)$ to (\ref{nonlinearproblem}). Furthermore, there holds
\begin{equation*}
\|\eta\|_*\leq C\ve^{1-\beta}	,
\end{equation*}
and $\eta_{\ve, \{ p_i \}}$ is continuous in $\{p_i\}$.
\end{proposition}


\section{The reduced problem and the multiple vortex rings solutions}\label{sec5}
\subsection{The reduced problem}

To find a real solution to problem (\ref{realequation}), we solve the reduced problem by finding the positions of the vortex points $\{p_i\}$ such that the coefficients $c_{\ell j}$ in (\ref{nonlinearproblem}) are zero for small $\ve $. In the previous section, we have deduced the existence of $\eta$ to the projected nonlinear problem:
\begin{equation*}
\mathbb{L}\eta+\frac{E_2}{\m U}+N(\eta)=\sum_{j}c_{j}\frac{\nabla u_j}{u_j}\tilde{\rho}_j(x).	
\end{equation*}
So $c_j=0$ is equivalent to
\begin{equation}\label{reduction}
Re\int_{\R^2}u_j[\mathcal{L}\eta+\frac{E_2}{\mathcal{U}}+ N(\eta)]	\nabla \bar{u}_jdx=0.
\end{equation}

By the relation of $u_j\mathbb{L}$ and $L_0(\phi_j)$ where $\phi_j=u_j\eta$,
\begin{equation*}
u_j\mathbb{L}(\eta)=L_0(\phi_j)+o(\frac{1}{\rho^2})\phi_j,	
\end{equation*}
where
\begin{equation*}
L_0(\phi)=\Delta \phi+(1-S^2)\phi-2Re(\bar{u}_0\phi)u_0.
\end{equation*}
One has, by integration by parts,
\begin{equation*}
\begin{split}
Re\int_{\R^2}u_j\mathbb{L}\eta \, \nabla \bar{u}_jdx&=Re\int_{\R^2}(L_0+O(\frac{1}{\rho^2})\phi_j\, \nabla \bar{u}_jdx\\
&=Re\int_{\R^2}\phi_j L_0(\nabla \bar{u}_j)dx+o(\ve)=o(\ve),
\end{split}	
\end{equation*}
and using the expression of $ N(\psi)$,
\begin{equation*}
Re\int_{\R^2}u_j N(\eta) \, \nabla \bar{u}_jdx=o(\ve)	.
\end{equation*}

We now compute
$$Re\int_{\R^2}u_j\frac{E_2}{\mathcal{U}}\nabla \bar{u}_jdx.$$
Recall that one can write $\mathcal{U}$ as $u_j\alpha_j$,
where $\alpha_j=\pi_{\ell \neq j}u_\ell e^{i\tilde{\varphi}_{\bf p}}$,
near each vortex point $p_j$. By (\ref{estimateofe21}) in Section \ref{sec3}, we have
\begin{equation*}
\begin{split}
&Re\int_{\R^2}\frac{E_{21}}{\alpha_j}{\nabla \bar{u}_j}dx\\
&=Re(i\ve |\ln \ve|)\int_{\R^2}\big[\frac{S'(r_j)}{S(r_j)}\frac{r_{j,2}}{r_j}+i\p_{x_2}\varphi_d\big]\Big(
\frac{S'(r_j)}{S(r_j)}\nabla r_j-i{\tau_j}\nabla \theta_j \Big)S^2(r_j)dx\\
&+o(\ve)\\
&=-\ve|\ln \ve|{\tau_j}\int_{\R^2}SS'(r)\Big( \p_{x_2}\varphi_d-{\tau_j}\frac{x_2}{r}\nabla \theta \Big)dx+o(\ve)\\
&=	-\ve|\ln \ve|{\tau_j}\int_{\R^2}SS'(r)\Big( \frac{x_1}{r^2}\nabla r-\frac{x_2}{r}\nabla \theta \Big)dx+o(\ve)\\
&=(-\pi\ve|\ln\ve|{\tau_j},\, 0)+o(\ve),
\end{split}	
\end{equation*}
where we have used the estimate (\ref{asymp3}).
On the other hand, by (\ref{estimateofe22}),
\begin{equation*}
\begin{split}
&Re\int_{\R^2}\frac{E_{22}}{\alpha_j}{\nabla \bar{u}}_jdx\\
&=\int_{\R^2}\Big(\sum_\ell\frac{1}{r_\ell^2}-|\nabla \varphi_d|^2+\frac{1}{x_1}\sum_\ell\frac{S'(r_\ell)}{S(r_\ell)}\p_{x_1}r_\ell   \Big)S(r_j)S'(r_j)\nabla r_j dx\\
&+2{\tau_j}\int \sum_\ell \frac{S'(r_\ell)}{S(r_\ell)}\nabla r_\ell\cdot \nabla \varphi_d \nabla \theta_j S^2(r_j)dx+o(\ve)\\
&=-\int_{\R^2}|\nabla \varphi_d|^2\nabla r_j S(r_j)S'(r_j) dx\\
&+\frac{1}{p_{j,1}}\int_{\R^2}(S'(r))^2\p_{x_1}r\nabla r dx\\
&+2{\tau_j}\int_{\R^2} \nabla r_j\cdot \nabla \varphi_d \nabla \theta_j S(r_j)S'(r_j)\, dx+o(\ve)\\
&=I_1+o(\ve).
\end{split}	
\end{equation*}


Recall the relation of $\varphi_d$ and $\psi$ in (\ref{relation1}), one has
\begin{equation*}
\nabla \varphi_d=\Big(\sum_{j=1}^{\m K}{\tau_j}\p_2 A_{p_j},\, \, -\sum_{j=1}^{\m K}{\tau_j}(\frac{A_{p_j}}{x_1}+\p_1 A_{p_j} )  \Big).	
\end{equation*}
It has been shown in \cite{Jackson} that
\begin{equation*}
A_a(x_1, x_2)=\sqrt{\frac{a_1}{x_1}}\frac{1}{\kappa}\Big[ (2-\kappa^2)K(\kappa^2)-2E(\kappa^2) \Big],	
\end{equation*}
where
\begin{equation*}
\kappa^2(x)=\frac{4a_1x_1}{x_1^2+a_1^2+(x_2-a_2)^2+2a_1x_1}	
\end{equation*}
and $K, \, E$ are the complete elliptic integrals of first and second kind, i.e.,
\begin{equation*}
\begin{split}
K(s)&=\int_{0}^{\frac{\pi}{2}}(1-s\sin^2\theta)^{-\frac{1}{2}}d\theta, 	\\
 E(s)&=\int_{0}^{\frac{\pi}{2}}(1-s\sin^2\theta)^{\frac{1}{2}}d\theta.
 \end{split}
\end{equation*}
They satisfy
\begin{equation*}
\begin{split}
K'(s)=K(1-s), \, \, E'(s)=E(1-s) \mbox{ for }1<s<1.
\end{split}	
\end{equation*}
Note that
 $
A_{\lambda a}(\lambda x)=A_a(x)	$,
and  for $s\to 1$,
\begin{equation*}
\begin{split}
K(s)&=-\frac{1}{2}\ln (1-s)(1+\frac{1-s}{4})+\ln 4+O(1-s),\\
E(s)&=1-\ln (1-s)\frac{1-s}{4}+O(1-s).	
\end{split}	
\end{equation*}
Moreover, as we mentioned before,  when $r=|z-a|=o(|a_1|)$, 	
\begin{equation*}
A_a(z)=\Big( \ln \frac{a_1}{r}+3\log(2)-2 \Big)+O\Big( \frac{r}{a_1}|\ln \frac{r}{a_1}| \Big)
\end{equation*}
and
\begin{equation*}
\p_r A_a=-\frac{1}{r}+O(\frac{1}{a_1}). 	
\end{equation*}
Combining all these, one has
\begin{equation*}
\begin{split}
&-\int_{\R^2}|\nabla \varphi_d|^2\nabla r_j S(r_j)S'(r_j) dx\\
&+2{\tau_j}\int_{\R^2} \nabla r_j\cdot \nabla \varphi_d \nabla \theta_j S(r_j)S'(r_j)\, dx\\
&=-2{\tau_j}\int_{\R^2}S(r_j)S'(r_j)\big(\nabla \theta_j\cdot\nabla \varphi_d \, \nabla r_j -\nabla r_j\cdot \nabla \varphi_d \, \nabla \theta_j \big)+o(\ve)\\
&=2{\tau_j}\Big(\frac{{\tau_j}}{p_{j,1}}\int \frac{S(r_j)S'(r_j)}{r_j}A_{p_j}(x)dx+\frac{\pi}{p_{j,1}}\sum_{\ell\neq j}{\tau_\ell} A_{p_\ell}(p_j), \, \, 0\Big)\\
&+2{\tau_j}\pi \sum_{\ell\neq j}{\tau_\ell}\nabla A_{p_\ell}(p_j)+o(\ve)\\
&=\frac{2\pi}{p_{j,1}}\Big(\ln p_{j,1}+c_0+ \sum_{\ell\neq j}{\tau_j\tau_\ell}A_{p_\ell}(p_j) , \, 0 \Big)+2\pi \sum_{\ell\neq j}{\tau_j\tau_\ell}\nabla A_{p_\ell}(p_j)+o(\ve)
\end{split}	
\end{equation*}
where
\begin{equation}\label{c0}
c_0=3\ln 2-2-\frac{1}{\pi}\int \frac{S(r)S'(r)\ln r}{r}dx.
\end{equation}
So one has
\begin{equation*}
I_1=\frac{2\pi}{p_{j,1}}\Big(\ln p_{j,1}+c_1+ \sum_{\ell\neq j}{\tau_j\tau_\ell}A_{p_\ell}(p_j) , \, 0 \Big)+2\pi \sum_{\ell\neq j}\tau_j\tau_\ell\nabla A_{p_\ell}(p_j)+o(\ve)	
\end{equation*}
where
\begin{equation*}
c_1=c_0+\frac{1}{2}\int_0^\infty S'(r)rdr.
\end{equation*}


By the above estimates,

\begin{equation*}
\begin{split}
&Re\int_{\R^2}\frac{E_2}{\alpha_j}\partial_{x_1}\bar{u}_j dx	\\
&=-\pi \Big[{\tau_j} \ve|\ln \ve|-\frac{2\ln p_{j,1}}{p_{j,1}}-\frac{2c_1}{p_{j,1}}-2\sum_{\ell\neq j}\tau_j\tau_\ell\frac{A_{p_\ell}(p_j)}{p_{j,1}}-2\sum_{\ell\neq j}\tau_j\tau_\ell\p_1 A_{p_\ell}(p_j)
\Big]\\
&+o(\ve)
\end{split}
\end{equation*}
and
\begin{equation*}
Re\int_{\R^2}\frac{E_2}{\alpha_j}\partial_{x_2}\bar{u}_j dx	=2\pi\sum_{\ell\neq j}\tau_j\tau_\ell\p_2 A_{p_\ell}(p_j)
  +o(\ve).
\end{equation*}

We now have the following reduced problem:

\begin{lemma}
The reduced problem (\ref{reduction}) is equivalent to the following system of the vortex points $\{p_j \}$:
\begin{equation}\label{reduced1}
\begin{split}
& {\tau_j}\ve|\ln \ve|-\frac{2\ln p_{j,1}}{p_{j,1}}-\frac{2c_1}{p_{j,1}}-2\sum_{\ell\neq j}\tau_j\tau_\ell\frac{A_{p_\ell}(p_j)}{p_{j,1}}-2\sum_{\ell\neq j}\tau_j\tau_\ell\p_1 A_{p_\ell}(p_j)
=o(\ve),
\end{split}
\end{equation}
and
\begin{equation}\label{reduced2}
2\sum_{\ell\neq j}\tau_j\tau_\ell\p_2 A_{p_\ell}(p_j) =o(\ve).
\end{equation}	
\end{lemma}

Using the scaling invariance
\begin{equation*}
A_{\lambda a}(\lambda x)=A_a(x),
\end{equation*}
if we denote by
\begin{equation*}
p_j=\frac{\tilde{p}_j}{\ve }, 	
\end{equation*}
where $|\tilde{p}_{j,1}|=O(1)$, we can get the reduced problem for $\tilde{p}_j$:

\begin{equation}\label{reduced11}
\begin{split}
& {\tau_j}|\ln \ve|+\frac{2 \ln\ve }{\tilde{p}_{j,1}}-\frac{2 \ln \tilde{p}_{j,1}}{\tilde{p}_{j,1}}
-\frac{2c_1}{\tilde{p}_{j,1}}-2\sum_{\ell\neq j}\tau_j\tau_\ell\frac{A_{\tilde{p}_\ell}(\tilde{p}_j)}{\tilde{p}_{j,1}}\\
&-2\sum_{\ell\neq j}\tau_j\tau_\ell\p_1 A_{\tilde{p}_\ell}(\tilde{p}_j)
=o(1),
\end{split}
\end{equation}
and
\begin{equation}\label{reduced22}
2\sum_{\ell\neq j}\tau_j\tau_\ell\p_2 A_{\tilde{p}_\ell}(\tilde{p}_j) =o(1).
\end{equation}
Using the asymptotic behavior of $A_a(x)$ and $A'_a(r)$, and recall that
\begin{equation*}
|\tilde{p}_\ell-\tilde{p}_j|\sim \frac{1}{\ln \ve }, \, |\tilde{p}_{j,1}|\sim O(1),	
\end{equation*}
we obtain the following  equivalent reduced problem:

\begin{equation}\label{finalreducedproblem}
\left\{\begin{array}{l}
 {\tau_j}|\ln \ve|+\frac{2 \ln\ve }{\tilde{p}_{j,1}}+2\sum_{\ell\neq j}\tau_j\tau_\ell\frac{\tilde{p}_{j,1}-\tilde{p}_{\ell,1}}{|\tilde{p}_j-\tilde{p}_\ell|^2}=o(\ln \ve),\\
2\sum_{\ell\neq j}\tau_j\tau_\ell\frac{\tilde{p}_{j,2}-\tilde{p}_{\ell,2}}{|\tilde{p}_j-\tilde{p}_\ell|^2} =o(1).	
\end{array}	
\right.
\end{equation}

\subsection{Vortex locations and their generating polynomials}

\label{sec6}

In this section, we construct a family of polynomials whose roots will
correspond to the locations of the vortex rings.

For each rescaled vortex point $\tilde{p}_{j},j=1,...,\mathcal{K},$ we have
associated a degree $\tau_{j}=\pm1.$ To analyze the reduced problem in a more
precise way, let us relabel those points with $\tau_{j}=1$ by $\tilde{p}%
_{1}^{+},...,\tilde{p}_{m}^{+}$ and those with $\tau=-1$ will be denoted by
$\tilde{p}_{1}^{-},...,\tilde{p}_{n}^{-}.$ We then write
\begin{align*}
\tilde{p}_{j}^{+}  & =\alpha_{0}+\alpha+\frac{1}{\left\vert \ln\varepsilon
	\right\vert }\mathbf{a}_{j},\text{ for }j=1,...,m,\\
\tilde{p}_{j}^{-}  & =\alpha_{0}+\alpha+\frac{1}{\left\vert \ln\varepsilon\right\vert
}\mathbf{b}_{j},\text{ for }j=1,...,n.
\end{align*}
Here $\alpha_{0}$ is a fixed constant only depends on $m,n,$ and
$\alpha=o\left(  1\right)  $ depends on $\varepsilon.$ Inserting these into
the reduced problem  (\ref{finalreducedproblem}), we find that, at the main order, $\left(  \mathbf{a}%
_{1},...,\mathbf{a}_{m},\mathbf{b}_{1},...,\mathbf{b}_{n}\right)  $ should
satisfy the following system:
\[
\left\{
\begin{array}
[c]{l}%
{\displaystyle\sum\limits_{j=1,j\neq k}^{m}}
\frac{1}{\mathbf{a}_{k}-\mathbf{a}_{j}}-%
{\displaystyle\sum\limits_{j=1}^{n}}
\frac{1}{\mathbf{a}_{k}-\mathbf{b}_{j}}=\frac{1}{2}-{\alpha_{0}}^{-1},\text{ for
}k=1,...,m,\\
-%
{\displaystyle\sum\limits_{j=1,j\neq k}^{n}}
\frac{1}{\mathbf{b}_{k}-\mathbf{b}_{j}}+%
{\displaystyle\sum\limits_{j=1}^{m}}
\frac{1}{\mathbf{b}_{k}-\mathbf{a}_{j}}=\frac{1}{2}+{\alpha_{0}}^{-1}, \text{ for
}k=1,...,n.
\end{array}
\right.
\]
This can be regarded as a balancing condition between the multiple vortex
rings. Adding together the $m+n$ equations in the balancing condition, we find
that a necessary condition for the existence of a balancing configuration is
$\left( {\alpha_{0}}^{-1}-\frac{1}{2}\right)  m+\left( {\alpha_{0}}^{-1}+\frac{1}%
{2}\right)  n=0.$ It follows that $\alpha_{0}=2\frac{m+n}{m-n}.$
Therefore, we are lead to consider the system
\begin{equation}
\left\{
\begin{array}
[c]{l}%
{\displaystyle\sum\limits_{j=1,j\neq k}^{m}}
\frac{1}{\mathbf{a}_{k}-\mathbf{a}_{j}}-%
{\displaystyle\sum\limits_{j=1}^{n}}
\frac{1}{\mathbf{a}_{k}-\mathbf{b}_{j}}=-n,\text{ for }k=1,...,m,\\
-%
{\displaystyle\sum\limits_{j=1,j\neq k}^{n}}
\frac{1}{\mathbf{b}_{k}-\mathbf{b}_{j}}+%
{\displaystyle\sum\limits_{j=1}^{m}}
\frac{1}{\mathbf{b}_{k}-\mathbf{a}_{j}}=-m, \text{ for }k=1,...,n.
\end{array}
\right.  \label{Balance}%
\end{equation}
To find solutions to this system, we define the generating polynomial as
\[
P\left(  x\right)  :=%
{\displaystyle\prod\limits_{j=1}^{m}}
\left(  x-\mathbf{a}_{j}\right)  ,\text{ \ }Q\left(  x\right)  :=%
{\displaystyle\prod\limits_{j=1}^{n}}
\left(  x-\mathbf{b}_{j}\right)  .
\]
If $\mathbf{a}_{j},\mathbf{b}_{j}$ satisfy $\left(  \ref{Balance}\right)  ,$
then
\begin{equation}
P^{\prime\prime}Q-2P^{\prime}Q^{\prime}+PQ^{\prime\prime}+nP^{\prime
}Q-mPQ^{\prime}=0.\label{PQ}%
\end{equation}

The case of $m=n$ has been studied in \cite{Liu-Wei} . In this case, the
system $\left(  \ref{Balance}\right)  $ is equivalent to
\[
\left\{
\begin{array}
[c]{l}%
{\displaystyle\sum\limits_{j=1,j\neq k}^{m}}
\frac{1}{\mathbf{a}_{k}-\mathbf{a}_{j}}-%
{\displaystyle\sum\limits_{j=1}^{n}}
\frac{1}{\mathbf{a}_{k}-\mathbf{b}_{j}}=-1,\text{ for }k=1,...,m,\\%
{\displaystyle\sum\limits_{j=1,j\neq k}^{n}}
\frac{1}{\mathbf{b}_{k}-\mathbf{b}_{j}}-%
{\displaystyle\sum\limits_{j=1}^{m}}
\frac{1}{\mathbf{b}_{k}-\mathbf{a}_{j}}=1, \text{ for }k=1,...,n.
\end{array}
\right.
\]
The polynomial solutions of this system are connected with theory of
integrable system. Indeed, letting $\phi=\frac{Q}{P}\exp\left(  x\right)  $
and $u=2\left(  \ln P\right)  ^{\prime\prime}.$ The equation $\left(
\ref{PQ}\right)  $ can be rewritten as
\[
\phi^{\prime\prime}+u\phi=\phi.
\]
This equation appears as the first equation in the Lax pair of the KdV
equation and has the Darboux invariance property. The polynomial solutions of
$\left(  \ref{PQ}\right)  $ in this case are given by the Adler-Moser polynomials.

From the view point of numerical computation, the equation $\left(
\ref{PQ}\right)  $ is indeed easier than $\left(  \ref{Balance}\right)  .$
Note that our construction of multiple vortex ring solutions requires that all
the points $\mathbf{a}_{j},j=1,...,m$ and $\mathbf{b}_{j},j=1,...,n$ are
distinct from each other. Therefore we require that the polynomials $P$, $Q$
satisfy the following condition:

\medskip

(H1) $P,Q$ have no repeated roots.

\medskip

Our construction also requires the following condition:

\medskip

(H2) The set of points $\left\{  \mathbf{a}_{1},\cdots,\mathbf{a}_{m},\text{
}\mathbf{b}_{1},\cdots,\mathbf{b}_{n}\right\}  $ are symmetric with respect to
the $x_{1}$ axis.

\medskip

Observe that equation $\left(  \ref{PQ}\right)  $ implies that if $X_{0}$ is a
common root of $P$ and $Q,$ then necessarily $X_{0}$ is a repeated root of $P$
or $Q.$

We observe that due to the translation invariance of the equation in the
balancing condition, we can normalize the polynomials $P,$ $Q$ as
\begin{align*}
P\left(  x\right)   &  =s_{1}+s_{2}x+...+s_{m-1}x^{m-1}+x^{m},\\
Q\left(  x\right)   &  =t_{1}+t_{2}x+...+t_{n-2}x^{n-2}+x^{n}.
\end{align*}
That is, the $x^{n-1}$ term in $Q$ can be chosen to be zero.
In this section, we would like to find some solution pair $(P,Q)$ using software such as Maple. Then in the next section, we shall use techniques of integrable system to find a sequence of solution pairs, with explicit Wronskian representation.

Let us consider the
case of $m+n\leq12.$ With this constraints, we find, using Maple, that there exist
polynomial solutions to $\left(  \ref{PQ}\right)  $ satisfying (H1) and whose
roots satisfy (H2), if further $\left(  m,n\right)  $ are one of the cases in
the set
\[
S:=\{\left(  2,1\right)  ,\left(  3,2\right)  ,\left(  4,3\right)  ,\left(
5,4\right)  ,\left(  6,5\right)  \}.
\]
Indeed, if $\left(  m,n\right)  =\left(  2,1\right)  ,$ then $\left(
\ref{PQ}\right)  $ has a solution of the form
\[
P\left(  x\right)  =x^{2}-2x+2,\text{ }Q\left(  x\right)  =x.
\]
If $\left(  m,n\right)  =\left(  3,2\right)  ,$ then $\left(  \ref{PQ}\right)
$ has solution:
\[
P\left(  x\right)  =x^{3}-2x^{2}+\frac{7}{2}x-\frac{3}{2},\text{ }Q\left(
x\right)  =x^{2}+1.
\]
If $\left(  m,n\right)  =\left(  4,3\right)  ,$ then $\left(  \ref{PQ}\right)
$ has solution:
\begin{align*}
P\left(  x\right)   &  =x^{4}-2x^{3}+\frac{44}{9}x^{2}-\frac{89}{27}%
x+\frac{533}{324},\\
\text{ }Q\left(  x\right)   &  =x^{3}+\frac{13}{6}x+\frac{13}{54}.
\end{align*}
If $\left(  m,n\right)  =\left(  5,4\right)  ,$ then $\left(  \ref{PQ}\right)
$ has solution:
\begin{align*}
P\left(  x\right)   &  =x^{5}-2x^{4}+\frac{449}{72}x^{3}-\frac{749}{144}%
x^{2}+\frac{12919}{2592}x-\frac{16015}{15552},\\
\text{ }Q\left(  x\right)   &  =x^{4}+\frac{61}{18}x^{2}+\frac{16}{27}%
x+\frac{1337}{1296}.
\end{align*}
When $\left(  m,n\right)  =\left(  6,5\right)  ,$ we have
\begin{align*}
P\left(  x\right)   &  =x^{6}-2x^{5}+\frac{2269}{300}x^{4}-\frac
{193279}{27000}x^{3}+\frac{10810499}{1080000}x^{2}-\frac{57115601}%
{16200000}x+\frac{3980046413}{2916000000},\\
Q\left(  x\right)   &  =x^{5}+\frac{1669}{360}x^{3}+\frac{3607}{3600}%
x^{2}+\frac{1112099}{324000}x+\frac{23805769}{48600000}.
\end{align*}
The roots of $P,Q$ listed above are solutions of the balancing system. Here we
list them in the order $\mathbf{a}_{1},...,\mathbf{a}_{m},\mathbf{b}%
_{1},...,\mathbf{b}_{n}$ and denote it by $\mathcal{P}_{\left(  m,n\right)
}.$

The numerical value can be listed as below:%

\begin{align*}
\mathcal{P}_{\left(  2,1\right)  }  &  :\left(  1+i,1-i,0\right)  ,\\
\mathcal{P}_{\left(  3,2\right)  }  &  :\left(
0.56,0.72+1.48i,0.72-1.48i,i,-i\right)  ,\\
\mathcal{P}_{\left(  4,3\right)  }  &  :(
0.393-0.57i,0.393+0.57i,0.607-1.76i,0.607+1.76i,\\
&  -0.11,0.055-1.48i,0.055+1.48i),
\end{align*}%

\begin{align*}
\mathcal{P}_{\left(  5,4\right)  }  &
:(0.255,0.322-0.938i,0.322+0.938i,0.55-1.948i,0.55+1.948i,\\
&  -0.107-0.567i,-0.107+0.567i,0.107-1.758i,0.107+1.758i),
\end{align*}%
\begin{align*}
\mathcal{P}_{\left(  6,5\right)  }  &
:(0.191-0.395i,0.191+0.395i,0.29-1.2i,0.29+1.2i,0.52-2.09i,\\
&  0.52+2.09i,-0.145,-0.078-0.94i,-0.078+0.94i,0.15-1.95i,0.15+1.95i).
\end{align*}

\begin{figure}[ptb]
	\caption{$(m,n)=(2,1)$}
	\centering
	\includegraphics[
	height=9in,
	width=7in
	]{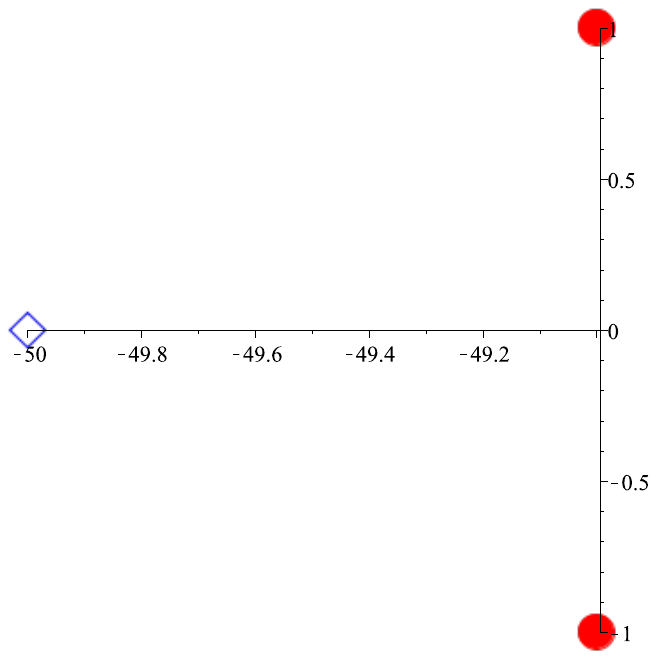}
\end{figure}

\begin{figure}[ptb]
	\caption{$(m,n)=(4,3)$}%
	\centering
	\includegraphics[
	height=9in,
	width=7in
	]{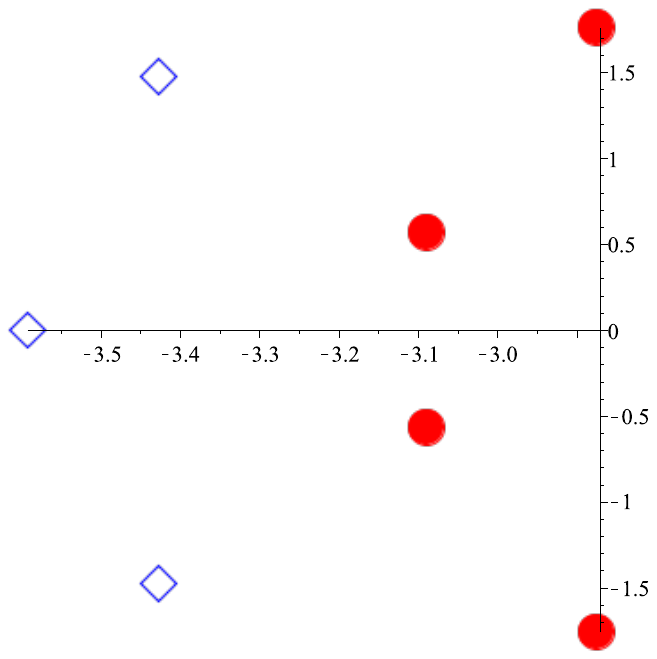}\end{figure}
\begin{figure}[ptb]
	\caption{$(m,n)=(6,5)$}%
	\centering
	\includegraphics[
	height=9in,
	width=7in
	]{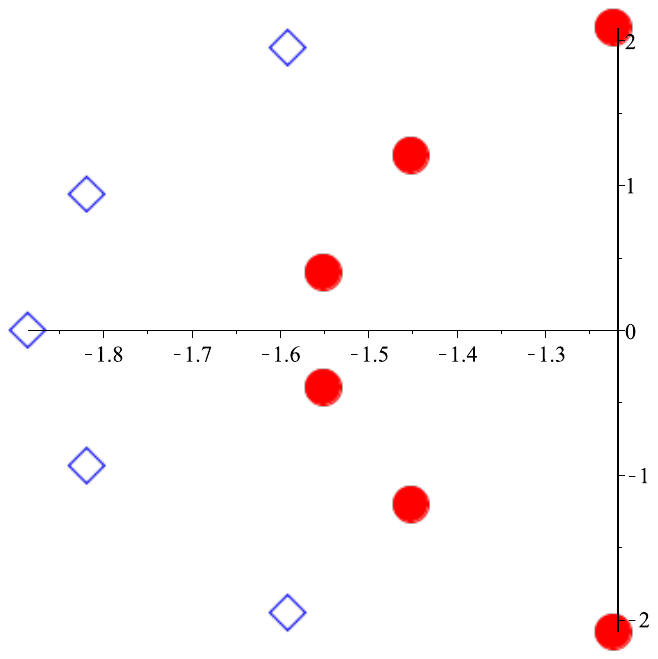}\end{figure}

 Let us denote the pair $(P,Q)$ for $(m,n)=(j,j-1)$ as $(P_j,Q_j)$. Then for the above examples, we can see that $P_j$ is simply a translation in the $x$ variable of $Q_{j+1}$. We will see in the next section that this is true for all $m=n+1$.

Next let us consider the linearized operator around the solution. Let us
denote the left hand side of the $j$-th equation of (\ref{Balance}) by
$F_{j}.$ Then we can compute the linearization $dF$ of the map
\[\label{mapF}
F:\left(  \mathbf{a}_{1},...,\mathbf{a}_{m},\mathbf{b}_{1},...,\mathbf{b}%
_{n}\right)  \rightarrow\left(  F_{1},...,F_{m+n}\right)  .
\]
$dF$ evaluated at the point $\mathcal{P}_{\left(  m,n\right)  }$ is a matrix,
which can be explicitly computed. The solvability of our original reduced
problem is closely related to the nondegeneracy of $dF.$ Since any translation
of the $\left(  \mathbf{a},\mathbf{b}\right)  $ is still a solution to the
balancing system, necessarily the determinant of this matrix is zero. That is,
$0$ is an eigenvalue of $dF$. Observe that $\left(  1,1,....,1\right)  $ is an
eigenvector. We call $\left(  \mathbf{a}_{1},...,\mathbf{a}_{m},\mathbf{b}%
_{1},...,\mathbf{b}_{n}\right)  $ nondegenerated, if the kernel of $dF$ is one
dimensional. One can check by explicit computations that for the solutions
$\mathcal{P}_{\left(  m,n\right)  }$ listed above, they are all
nondegenerated.

It is worth pointing out that if $\left(  m,n\right)  $ is not in $S,$ there
may still have polynomials $P,Q$ satisfying $\left(  \ref{PQ}\right)  ,$ but
with repeated roots. For instance, when $\left(  m,n\right)  =\left(
4,1\right)  ,$ it has a solution with
\[
P\left(  x\right)  =x^{4}+4x^{3},\text{ \ }Q\left(  x\right)  =x.
\]
When $\left(  m,n\right)  =\left(  5,3\right)  ,$ it has a solution with
\[
P\left(  x\right)  =x^{5}-\frac{4}{3}x^{4}+\frac{4}{3}x^{3}-\frac{8}{9}%
x^{2}+\frac{8}{27}x,\text{ }Q\left(  x\right)  =x^{3}.
\]

A given pair $\left(  m,n\right)  $ can be used in the construction of multiple
vortex rings, if there exist polynomial solutions to $\left(  \ref{PQ}\right)
$ satisfying (H1) and (H2). In this respect, there are many questions remain
to be answered. For instances, are there infinitely many such pairs? If
$\left(  m,n\right)  $ is such a pair, is it necessarily that $m=n+1?$ Is the
balancing configuration unique up to translation? These questions will be partially answered in the next section.

Now let us come back to our original reduced problem (\ref{finalreducedproblem}) of the GP equation. For
each $\left(  m,n\right)  \in S.$ We have a special solution $\left(  \mathbf{a}%
_{1}^{0},...,\mathbf{a}_{m}^{0},\mathbf{b}_{1}^{0},...,\mathbf{b}_{n}%
^{0}\right)  $ given by $\mathcal{P}_{\left(  m,n\right)  }.$ If we define vector
$\beta$ by
\begin{align*}
\mathbf{a}_{j}  & =\mathbf{a}_{j}^{0}+\beta_{j},j=1,...,m,\\
\mathbf{b}_{j}  & =\mathbf{b}_{j}^{0}+\beta_{j+m},j=1,...,n,
\end{align*}
then the reduced problem (\ref{finalreducedproblem}) takes the form
\begin{equation}
dF\left(  \beta\right)  =G\left(  \alpha,\beta\right)  +\alpha \alpha_0^{-2}\mathbf{e}%
_{1}\mathbf{,}\label{redu}%
\end{equation}
where $G\left(  \alpha,\beta\right)  =o\left(  1\right)  $ as $\varepsilon
\rightarrow0,$ with higher order dependence on $\alpha,\beta,$ and
$\mathbf{e}_{1}$ is a $m+n$ dimensional column vector whose first $m$ entries
are  equal to $-1$ and the last $n$ entries are all equal to $1.$ Note that
$dF$ is in general not a symmetric matrix. However, since $dF$ is
nondegenerated, the kernel of $\left(  dF\right)  ^{T}$ is spanned by
$\mathbf{e}_{2}:=\left(  1,...,1\right)  .$ Using the fact that $m-n=1,$ \ we
find that the projection of the right hand side of $\left(  \ref{redu}\right)
$ onto $\mathbf{e}_{2}$ is equal to $G\cdot\mathbf{e}_{2}-\alpha\alpha_0^{-2}.$  Now let us consider the projected
problem
\begin{equation}
dF\left(  \beta\right)  =G\left(  \alpha,\beta\right)  +\alpha \alpha_0^{-2}\mathbf{e}%
_{1}-\frac{G\cdot\mathbf{e}_{2}-\alpha\alpha_0^{-2}}{m+n}\mathbf{e}_{2}.\label{pro}%
\end{equation}
Note that for each fixed small $\alpha,$ using the nondegeneracy of the solution $\left(
\mathbf{a}_{1}^{0},...,\mathbf{a}_{m}^{0},\mathbf{b}_{1}^{0},...,\mathbf{b}%
_{n}^{0}\right)  ,$ the projected system $\left(  \ref{pro}\right)  $ can be
solved and a solution $\beta$ depending on $\alpha.$ With this $\beta,$ we
then can solve the equation $G\cdot\mathbf{e}_{2}-\alpha\alpha_0^{-2}=0$ by a contraction
mapping argument. Hence the reduced problem $\left(  \ref{redu}\right)  $ can
be finally solved. Once this is done, with the help of linear theory of
Section 4, arguments similar as that of \cite{Liu-Wei} yield a solution to the
GP equation, satisfying the conclusion of Theorem \ref{thm1} .

\section{Recurrence relations and Wronskian representation of the generating
	polynomials}

In this section, we show that the generating polynomials of the balancing
system discussed in the previous section have  recurrence relations in the
case of $m=n+1$, and can be explicitly written down using certain Wronskians. The
main result of this section is the following

\begin{theorem}
	\label{Theor}There exists a sequence of polynomials $\mathcal{P}_{n},n=1,...,$
	such that
	\begin{equation}
	\mathcal{P}_{n+1}^{\prime\prime}\mathcal{P}_{n}-2\mathcal{P}_{n+1}^{\prime
	}\mathcal{P}_{n}^{\prime}+\mathcal{P}_{n+1}^{\prime\prime}\mathcal{P}%
	_{n}+n\mathcal{P}_{n+1}^{\prime}\mathcal{P}_{n}-\left(  n+1\right)
	\mathcal{P}_{n+1}\mathcal{P}_{n}^{\prime}=0,\label{e2}%
	\end{equation}
	where $\mathcal{P}_{n}$ is of degree $n$, $\mathcal{P}_{1}=x$ and
	$\mathcal{P}_{2}=x^{2}-2x+2.$ Moreover, up to a constant factor(see $\left(
	\ref{cons}\right)  $), these polynomials can be written as
	\[
	\exp\left(  -\frac{n\left(  n-1\right)  x}{2}\right)  W\left(  \omega
	_{1},...,\omega_{n}\right)  ,
	\]
	where $W$ represents the Wronskian, $\omega_{j}=\left(  x-a_{j}\right)
	\exp\left(  \left(  j-1\right)  x\right)  ,$ and $a_{1}=0,$ $a_{j+1}%
	=a_{j}+\frac{2}{j}.$
\end{theorem}

These polynomials can be regarded as a generalization of the Adler-Moser
polynomials. There are other types of generalization of the Adler-Moser
polynomials, see, for instance \cite{Lou}. We also refer to \cite{Aref2,Aref3, Cla, Liu-Wei} and the references cited therein for more discussion in this direction.

Recall that in the previous section, we derived the equation
\begin{equation}
P^{\prime\prime}Q-2P^{\prime}Q^{\prime}+PQ^{\prime\prime}+nP^{\prime}Q-\left(
n+1\right)  PQ^{\prime}=0. \label{e1}%
\end{equation}
For $n=1,$ we have found that $P\left(  x\right)  =x^{2}-2x+2,$ $Q\left(
x\right)  =x$ is a solution. To solve this equation for general $n,$ we define
$\phi=\frac{Q}{P}.$ Direction computation shows that the equation $\left(
\ref{e1}\right)  $ can be written as
\begin{equation}
\phi^{\prime\prime}+\left(  2\left(  \ln P\right)  ^{\prime\prime}-\left(  \ln
P\right)  ^{\prime}\right)  \phi-\left(  n+1\right)  \phi^{\prime}=0.
\label{e3}%
\end{equation}
Note that the equation in this form is different from the one considered by
Adler-Moser, in the sense that we have two additional terms corresponding to
$\left(  \ln P\right)  ^{\prime}$ and $\phi^{\prime}.$ Moreover, equation (\ref{e1}) is not of the standard Hirota bilinear form. This significantly
complicates the analysis.

For $n=1,$ we already know that equation $\left(  \ref{e3}\right)  $ has the
solution
\[
\bar{\phi}\left(  x\right)  :=\frac{Q}{P}=\frac{x}{x^{2}-2x+2}.
\]
It is worth pointing out, although not necessarily relevant to our later
analysis, $\bar{\phi}$ is smooth in the whole line. Equation $\left(
\ref{e3}\right)  $ is a second order ODE, it has another solution linearly
independent with $\bar{\phi}.$ One can check that $\phi^{\ast}$ defined below
is such a solution. Explicitly,
\[
\phi^{\ast}\left(  x\right)  :=\frac{2x^{3}-10x^{2}+21x-16}{x^{2}-2x+2}%
\exp\left(  2x\right)  .
\]
Note that $\phi^{\ast}$ can also be written as%
\[
\phi^{\ast}=\left(  \int_{-\infty}^{x}\frac{\exp\left[  \left(  n+1\right)
	s\right]  }{\bar{\phi}^{2}}ds\right)  \bar{\phi}\left(  x\right)  .
\]

Next we discuss the generalized Darboux transformation adapted to equation
(\ref{e3}). The following result can be found in the last section of \cite{Mat}.

\begin{lemma}
	Suppose $\phi=\phi_{1}$ and $\phi=\phi_{2}$ are two solutions of the equation
	\[
	u_{2}\phi^{\prime\prime}+u_{1}\phi^{\prime}+u_{0}\phi=0.
	\]
	Then the functions $\tilde{\phi}:=\phi_{2}^{\prime}-\frac{\phi_{1}^{\prime
		}\phi_{2}}{\phi_{1}}$ satisfies
	\[
	\tilde{u}_{2}\tilde{\phi}^{\prime\prime}+\tilde{u}_{1}\tilde{\phi}^{\prime
	}+\tilde{u}_{0}\tilde{\phi}=0,
	\]
	where $\tilde{u}_{2}=u_{2},\tilde{u}_{1}=u_{1}+u_{2}^{\prime},\tilde{u}%
	_{0}=u_{0}+u_{1}^{\prime}+2u_{2}\left(  \ln\phi_{1}\right)  ^{\prime\prime
	}+u_{2}^{\prime}\left(  \ln\phi_{1}\right)  ^{\prime}.$
\end{lemma}

To apply this lemma, we write equation $\left(  \ref{e3}\right)  $ as
\[
e^{-x}\phi^{\prime\prime}+e^{-x}\left(  2\left(  \ln P\right)  ^{\prime\prime
}-\left(  \ln P\right)  ^{\prime}\right)  \phi-e^{-x}\left(  n+1\right)
\phi^{\prime}=0.
\]
Let us define the new potential
\begin{align*}
\tilde{u}_{0} &  :=e^{-x}\left(  2\left(  \ln P\right)  ^{\prime\prime
}-\left(  \ln P\right)  ^{\prime}\right)  +\left(  n+1\right)  e^{-x}%
+2e^{-x}\left(  \ln\phi^{\ast}\right)  ^{\prime\prime}-e^{-x}\left(  \ln
\phi^{\ast}\right)  ^{\prime}.\\
\tilde{u}_{1} &  =-e^{-x}\left(  n+1\right)  -e^{-x},
\end{align*}
and the new function
\[
\Phi_{1}:=\bar{\phi}^{\prime}-\frac{\phi^{\ast\prime}\bar{\phi}}{\phi^{\ast}%
}=\frac{W\left(  \phi^{\ast},\bar{\phi}\right)  }{\phi^{\ast}}=\frac{e^{2x}%
}{\phi^{\ast}}.
\]
Then using the generalized Darboux transformation described in the previous lemma, we have
\[
e^{-x}\Phi_{1}^{\prime\prime}+\tilde{u}_{0}\Phi_{1}+\tilde{u}_{1}\Phi
_{1}^{\prime}=0.
\]
That is,
\begin{equation}
e^{-x}\Phi_{1}^{\prime\prime}+e^{-x}\left(  2\left(  \ln P_{3}\right)
^{\prime\prime}-\left(  \ln P_{3}\right)  ^{\prime}\right)  \Phi_{1}-\left(
n+2\right)  e^{-x}\Phi_{1}^{\prime}=0,\label{eqp3}%
\end{equation}
where the polynomial $P_{3}$ is defined by
\[
P_{3}:=P\phi^{\ast}e^{-2x}=2x^{3}-10x^{2}+21x-16.
\]
Equation $\left(  \ref{eqp3}\right)  $ precisely has the form $\left(
\ref{e3}\right)  .$ An important property is that the equation $\left(
\ref{eqp3}\right)  $ has another solution
\[
\Phi_{1}^{\ast}:=\frac{36x^{4}-312x^{3}+1136x^{2}-1972x+1357}{2x^{3}%
	-10x^{2}+21x-16}e^{3x}.
\]

The computations tell us that if $\mathcal{P}_{n}$ is a sequence of
polynomials satisfies the conclusion of Theorem \ref{Theor}, then we expect
the equation
\[
\phi^{\prime\prime}+\left(  2\left(  \ln\mathcal{P}_{n+1}\right)
^{\prime\prime}-\left(  \ln\mathcal{P}_{n+1}\right)  ^{\prime}\right)
\phi-\left(  n+1\right)  \phi^{\prime}=0,
\]
has two linearly independent solutions, of the form
\[
\phi_{1}=\frac{\mathcal{P}_{n}}{\mathcal{P}_{n+1}},\phi_{2}=\frac
{\mathcal{P}_{n+2}}{\mathcal{P}_{n+1}}e^{\left(  n+1\right)  x}.
\]
The Wronskian $W\left(  \phi_{1},\phi_{2}\right)  $ should be equal to
$ce^{\left(  n+1\right)  x}$ for some constant $c.$ Hence we get the following
recursive relations between $\mathcal{P}_{n},\mathcal{P}_{n+1},\mathcal{P}%
_{n+2}:$%
\[
\left(  \frac{\mathcal{P}_{n}}{\mathcal{P}_{n+1}}\right)  ^{\prime}%
\frac{\mathcal{P}_{n+2}}{\mathcal{P}_{n+1}}e^{\left(  n+1\right)  x}-\left(
\frac{\mathcal{P}_{n}}{\mathcal{P}_{n+1}}\right)  \left(  \frac{\mathcal{P}%
	_{n+2}}{\mathcal{P}_{n+1}}e^{\left(  n+1\right)  x}\right)  ^{\prime
}=ce^{\left(  n+1\right)  x}.
\]
That is,
\[
\left(  \mathcal{P}_{n}^{\prime}\mathcal{P}_{n+1}-\mathcal{P}_{n}%
\mathcal{P}_{n+1}^{\prime}\right)  \mathcal{P}_{n+2}-\mathcal{P}_{n}\left(
\mathcal{P}_{n+2}^{\prime}\mathcal{P}_{n+1}-\mathcal{P}_{n+2}\mathcal{P}%
_{n+1}^{\prime}+\left(  n+1\right)  \mathcal{P}_{n+2}\mathcal{P}_{n+1}\right)
=c\mathcal{P}_{n+1}^{3}.
\]
This can be written as
\[
\mathcal{P}_{n}^{\prime}\mathcal{P}_{n+2}-\mathcal{P}_{n}P_{n+2}^{\prime
}-\left(  n+1\right)  \mathcal{P}_{n}\mathcal{P}_{n+2}=c\mathcal{P}_{n+1}^{2}.
\]
If we normalize the polynomials $\mathcal{P}_{n}$ such that the highest order
term is $x^{n}.$ Then the constant $c$ satisfies
\[
-\left(  n+1\right)  =c
\]
We get the following recurrence relations
\begin{equation}
\mathcal{P}_{n}^{\prime}\mathcal{P}_{n+2}-\mathcal{P}_{n}P_{n+2}^{\prime
}-\left(  n+1\right)  \mathcal{P}_{n}\mathcal{P}_{n+2}+\left(  n+1\right)
\mathcal{P}_{n+1}^{2}=0.\label{re}%
\end{equation}

When we are given $\mathcal{P}_{n},\mathcal{P}_{n+1},$ the recurrence equation
$\left(  \ref{re}\right)  $ can be integrated, and we expect that the resulted
function $\mathcal{P}_{n+2}$ is a polynomial. However, in this step, we will
not get a free parameter in this polynomial, because solution of the
homogeneous equation has an exponential factor.

To show that integrating $\left(  \ref{re}\right)  $ indeed yields a
polynomial, we proceed to find the explicit formula of the sequence
$\mathcal{P}_{n}$ which satisfies $\left(  \ref{re}\right)  .$

Let us consider the sequence $a_{j}$ defined through the recurrence $a_{1}=0,$ and
$a_{j+1}=a_{j}+\frac{2}{j}.$ Let us define functions $\omega_{j}$ by
\begin{equation}
\omega_{j}=\left(  x-a_{j}\right)  \exp\left(  \left(  j-1\right)  x\right)
.\label{omega}%
\end{equation}
Then we define functions $\mathcal{P}_{n}$ through the Wronskian
\begin{equation}
\mathcal{P}_{n}:=c_{n}\exp\left(  -\frac{n\left(  n-1\right)  }{2}x\right)
W\left(  \omega_{1},...,\omega_{n}\right)  ,\label{Wron}%
\end{equation}
where
\begin{equation}
c_{n}=\left[  \left(  n-1\right)  !%
{\displaystyle\prod\limits_{1\leq i<j\leq n-1}}
\left(  j-i\right)  \right]  ^{-1}.\label{cons}%
\end{equation}
The normalizing constant $c_{n}$ is used to ensure that the highest order term
of $\mathcal{P}_{n}$ is $x^{n}.$ Note that $\mathcal{P}_{n}$ defined by
$\left(  \ref{Wron}\right)  $ are indeed polynomials of degree $n,$ and its
leading coefficient is a determinant of Vandermont type.

\begin{lemma}
	\label{threeterm}$\mathcal{P}_{n}$ defined by $\left(  \ref{Wron}\right)  $
	satisfies the three-term recurrence relation $\left(  \ref{re}\right)  .$
\end{lemma}

\begin{proof}
	For national simplicity, we write $W\left(  \omega_{1},...,\omega_{k}\right)
	$ as $W_{k}.$ Using $\left(  \ref{Wron}\right)  $ and the fact that
	\[
	\frac{c_{n+1}^{2}}{c_{n}c_{n+2}}=n+1,
	\]
	we see that to prove $\left(  \ref{re}\right)  ,$ it suffices to prove%
	\begin{equation}
	W_{n}^{\prime}W_{n+2}-W_{n}W_{n+2}^{\prime}+nW_{n}W_{n+2}+\left(  n+1\right)
	^{2}e^{x}W_{n+1}^{2}=0.\label{refi}%
	\end{equation}
	Following Adler-Moser \cite{Adler}, for any function $\xi,$ we define
	\[
	W_{k}\left(  \xi\right)  :=W\left(  \omega_{1},...,\omega_{k},\xi\right)  .
	\]
	Then we have the Jacobi identity(see \cite{Adler}, Lemma 1)
	\begin{equation}
	\left(  W_{k}\left(  \xi\right)  \right)  ^{\prime}W_{k+1}-W_{k}\left(
	\xi\right)  W_{k+1}^{\prime}-W_{k+1}\left(  \xi\right)  W_{k}=0\label{Jaco}%
	\end{equation}
	Direct computation tells us that%
	\[
	\omega_{j+1}^{\prime\prime}=j^{2}\omega_{j}\exp\left(  x\right)  .
	\]
	Using this relation and its differentiation and the fact that $\omega_1=x$, we obtain
	\[
	W_{k}\left(  1\right)  =\left(  -1\right)  ^{k}\left(  \left(  k-1\right)
	!\right)  ^{2}\exp\left[  \left(  k-1\right)  x\right]  W_{k-1}.
	\]
	We then compute
	\begin{align*}
	& \left(  -1\right)  ^{k}\left(  \left(  W_{k}\left(  1\right)  \right)
	^{\prime}W_{k+1}-W_{k}\left(  1\right)  W_{k+1}^{\prime}-W_{k+1}\left(
	1\right)  W_{k}\right)  \\
	& =\left(  \left[  \left(  k-1\right)  !\right]  ^{2}\exp\left[  \left(
	k-1\right)  x\right]  W_{k-1}\right)  ^{\prime}W_{k+1}\\
	& -\left(  \left(  k-1\right)  !\right)  ^{2}\exp\left[  \left(  k-1\right)
	x\right]  W_{k-1}W_{k+1}^{\prime}+\left(  k!\right)  ^{2}\exp\left[
	kx\right]  W_{k}^{2}.
	\end{align*}
	Dividing the right hand side by $\left[  \left(  k-1\right)  !\right]
	^{2}\exp\left(  k-1\right)  x,$ we get
	\[
	W_{k-1}^{\prime}W_{k+1}-W_{k-1}W_{k+1}^{\prime}+\left(  k-1\right)W_{k-1}W_{k+1}+k^{2}\exp\left(
	x\right)  W_{k}^{2}.
	\]
	By the Jacobi identity, this has to be zero. Letting $k=n+1,$ we get $\left(
	\ref{refi}\right)  .$ This finishes the proof.
\end{proof}

The conclusion of Theorem \ref{Theor} follows immediately from Lemma
\ref{threeterm} and the Darboux invariance property discussed above. Hence we
have abundant candidates of balancing configurations of multiple vortex rings.
In principle, the nondegeneracy of these configuration could be proved using
similar idea as that of \cite{Liu-Wei}. We leave this to a further study.

Finally, let us comment on the reason why we restrict to the case $m=n+1.$
Indeed, our original equation in Section 5 to be solved is%

\begin{equation}
P^{\prime\prime}Q-2P^{\prime}Q^{\prime}+PQ^{\prime\prime}+nP^{\prime
}Q-mPQ^{\prime}=0\label{mn}%
\end{equation}
Let $Q=x,$ $n=1$ and $m\geq3.$ Then the degree $m$ polynomial $P$ satisfying
$\left(  \ref{mn}\right)  $ necessarily has the factor $x^{3}.$ Hence $P$ and
$Q$ has a common root and can't be used in our construction. We conjecture that when $m-n>1$, there will be no balancing configurations satisfying our requirements stated in Section 5.

\bibliographystyle{plain}

\bigskip

\author{\noindent Weiwei Ao\\School of Mathematics and Statistics\\Wuhan University, Wuhan, Hubei, China \\Email: wwao@whu.edu.cn
	\\
	\text{}
	\\	Yehui Huang\\School of Mathematics and Physics, \\North China Electric Power University, Beijing, China,\\Email: yhhuang@ncepu.edu.cn
	\\
	\text{}
	\\	Yong Liu\\Department of Mathematics, \\University of Science and Technology of China, Hefei, China,\\Email: yliumath@ustc.edu.cn
	\\
	\text{}
	\\  Juncheng Wei\\Department of Mathematics, \\University of British Columbia, Vancouver, B.C., Canada, V6T 1Z2\\Email: jcwei@math.ubc.ca}

\end{document}